\documentclass[twocolumn]{autartmod}
\usepackage[utf8]{inputenc}
\usepackage[T1]{fontenc}
\usepackage{mathpazo}
\usepackage{graphicx}
\usepackage{placeins}
\usepackage{amsmath,amssymb,bm}
\usepackage{xcolor}
\usepackage{silence}
\usepackage{subcaption}

\makeatletter
\renewcommand*\l@subsection{\@dottedtocline{2}{2.5em}{2.3em}} 
\makeatother

\allowdisplaybreaks 
\usepackage{ragged2e} 

\usepackage{pgfplots}
\pgfplotsset{compat=newest}
\usetikzlibrary{arrows.meta,pgfplots.fillbetween}
\newcommand{\coloredblock}[3][0.2cm]{\tikz[baseline=-0.75ex]{\node[fill=#2, draw=#3, minimum size=#1, inner sep=0pt] (n) {};}}

\colorlet{fadedred}{red!30!gray!80!white}
\colorlet{fadedblue}{blue!70!gray!80!white}
\colorlet{fadedgreen}{green!60!blue!80!gray}

\definecolor{beige}{RGB}{245,245,220}

\usetikzlibrary{shadows}
\usepackage[many]{tcolorbox}
\newtcolorbox{eqbox}{
  colback=beige!30,     
  colframe=black!30,    
  boxrule=0.7pt,        
  arc=5pt,              
  boxsep=0pt,           
  left=0pt, right=5pt,  
  top=-6pt, bottom=-10pt
}
\newtcolorbox{eqboxb}{
  colback=beige!30,     
  colframe=black!30,    
  boxrule=0.7pt,        
  arc=5pt,              
  boxsep=0pt,           
  left=0pt, right=5pt,  
  top=-6pt, bottom=5pt
}

\usepackage{hyperref}
\hypersetup{
    colorlinks=true,
    linkcolor=fadedblue, 
    citecolor=fadedgreen,
    urlcolor=fadedgreen,
    pdftitle={CT-STL}
}

\renewcommand{\dot}[1]{\overset{\text{{\large{.}}}}{#1}}

\usepackage{comment}
\usepackage{todonotes}
\usepackage{cite}
\usepackage{textcomp}

\usepackage{empheq}
\usepackage{float}

\usepackage{graphicx}

\newtheorem{definition}{Definition}

\newtheorem{proof}{Proof}

\renewenvironment{proof}{\par\textcolor{gray}{\textit{\textbf{Proof:}~}}}{\hfill$\square$\par}

\usepackage{array}

\usepackage{enumerate, mdwlist}

\usepackage{booktabs}

\newcommand{\always}{\texttt{always}}
\newcommand{\eventually}{\texttt{eventually}}
\newcommand{\implication}{\texttt{implication}}
\newcommand{\conjunction}{\texttt{conjunction}}
\newcommand{\disjunction}{\texttt{disjunction}}
\newcommand{\Negation}{\texttt{Negation}}
\newcommand{\negation}{\texttt{negation}}
\newcommand{\until}{\texttt{until}}
\newcommand{\stland}{\texttt{and}}
\newcommand{\stlor}{\texttt{or}}

\newcommand{\true}{\texttt{true}}
\newcommand{\false}{\texttt{false}}

\newcommand{\relu}[1]{|{#1}|_+}
\newcommand{\negrelu}[1]{|{#1}|_-}

\newcommand{\defeq}{\vcentcolon=}

\edef\endfrontmatter{%
  \unexpanded\expandafter{\endfrontmatter}
  \noexpand\endNoHyper 
}
\begin{document}
\begin{frontmatter}
\runtitle{CT-STL Trajectory Optimization}
\vspace{-0.3cm}
\title{%
Successive Convexification for Trajectory Optimization with Continuous-time Satisfaction of Signal Temporal Logic Specifications
\!\!\!\!\!
\thanksref{footnoteinfo}
}
\vspace{-0.25cm}
\thanks[footnoteinfo]{
Corresponding author: Samet Uzun.}
\vspace{-0.38cm}
\author{Samet Uzun}$^{*}$\ead{samet@uw.edu},~~%
\author{Beh{\c{c}}et A{\c{c}}{\i}kme{\c{s}}e}$^{*}$\ead{behcet@uw.edu}%
\address{$^{*}$William E.\ Boeing Department of Aeronautics \& Astronautics, University of Washington, Seattle, WA, 98195}
\vspace{-0.2cm}
\vspace{-0.20cm}
\begin{abstract}
This paper presents a successive convexification framework for trajectory optimization under continuous-time Signal Temporal Logic (CT-STL) specifications. The framework employs generalized mean-based robustness (GMSR), a smooth and exact parameterization of discrete-time STL, as a logical building block for constructing differentiable CT-STL constraints in optimal control. It is integrated with time-dilation for free-final-time problems, finite-dimensional control parameterization, multiple-shooting discretization of the dynamics, and a convergence-guaranteed sequential convex programming method, prox-convex, to solve the nonconvex program.
The main CT-STL realization embeds temporal aggregation into augmented continuous-time dynamics. This augmentation-based construction is largely transcription-independent, can be incorporated into existing optimal-control pipelines with minimal structural changes, and enables smooth CT-STL parameterizations with accuracy controlled by a user-selected regularization parameter. We also discuss a complementary dense-time realization that evaluates CT-STL formulas directly on the integration subnodes used for dynamics discretization, yielding a smooth and exact parameterization on the numerical trajectory representation, up to the accuracy of the integration scheme.
The proposed GMSR-based formulations mitigate the locality and gradient-masking behavior of standard quantitative semantics and therefore provide a favorable landscape for gradient-based trajectory optimization. The framework is demonstrated through trajectory-optimization examples for a double-integrator system with continuous-time \always{}, \eventually{}, and \until{} specifications, and a 6-DoF quadrotor flight problem with combined \always{}, \implication{}, and \eventually{}-type specifications.
The implementation is available at \url{https://github.com/UW-ACL/TrajOpt_CT-STL}.
\end{abstract}
\end{frontmatter}
\section{Introduction}
\vspace{-0.3cm}
Temporal logic provides a systematic language for specifying complex task requirements in dynamical systems~\cite{baier2008principles}. Its expressive power makes it possible to encode sequencing, timing, persistence, implication, and conditional behaviors in a mathematically precise manner. As a result, temporal logic has become an important tool in planning and control,
with applications spanning autonomous driving~\cite{sahin2020autonomous}, robotic manipulation~\cite{kurtz2020trajectory}, multi-agent coordination~\cite{djeumou2022probabilistic}, mobile robotics~\cite{sahin2019multirobot,buyukkocak2021planning,cardona2024planning}, and UAV mission planning~\cite{pant2018fly,bacspinar2019mission}. Temporal logic constraints have also been incorporated into learning frameworks to enforce structural and behavioral properties in classification, prediction, preference learning, and policy synthesis~\cite{bombara2016decision,karagulle2024safe,ma2020stlnet,aksaray2016q,alshiekh2018safe}. These developments have created a strong demand for formulations that integrate temporal logic specifications directly into modern optimization and learning pipelines~\cite{leung2023backpropagation}.
Among temporal logics, linear temporal logic (LTL) has played a central role in formal synthesis for discrete-event systems~\cite{pnueli1977temporal}. Classical approaches rely on automata-theoretic constructions~\cite{alur1994theory}, together with receding-horizon variants and supporting software tools~\cite{wongpiromsarn2012receding,wongpiromsarn2011tulip}. Mixed-integer formulations have also been developed to encode temporal specifications directly into optimization problems~\cite{karaman2008optimal}. For continuous and hybrid systems, however, specifications must capture not only logical structure but also quantitative timing and predicates over real-valued signals. Metric temporal logic (MTL)~\cite{koymans1990specifying} and metric interval temporal logic (MITL)~\cite{alur1996benefits} enrich temporal logic with bounded timing intervals, while signal temporal logic (STL) augments these dense-time modalities with predicates over continuous-valued signals~\cite{maler2004monitoring}. This makes STL especially well suited to trajectory generation and control for physical systems.
A key step toward optimization-based synthesis under STL was the introduction of quantitative semantics. Robust semantics for MITL were proposed in~\cite{fainekos2009robustness}, and the now-standard space robustness (SR), together with time robustness, was introduced for STL in~\cite{donze2010robust}. These semantics assign a real-valued robustness score to a formula by recursively combining predicate values through $\min$ and $\max$ operators, thereby turning logical satisfaction into a numerical quantity that can be optimized~\cite{lindemann2018control,belta2019formal}.

\vspace{-0.2cm}
Most optimization-oriented work has focused on discrete-time STL (DT-STL), where the signal is evaluated only on a finite temporal grid~\cite{raman2014model,sadraddini2015robust,rodionova2021time,rodionova2022combined,pant2017smooth,haghighi2019control,gilpin2020smooth,mao2022successive,lindemann2019robust,mehdipour2019arithmetic,cardona2023mixed,mehdipour2020specifying,varnai2020robustness}. Existing DT-STL methods can be understood through a trade-off between exactness, smoothness, and numerical tractability. Mixed-integer encodings preserve Boolean correctness, but their worst-case complexity typically scales poorly with the number of binary variables~\cite{raman2014model,sadraddini2015robust,rodionova2021time,rodionova2022combined}. Direct use of SR avoids combinatorial structure, but the underlying $\min$ and $\max$ operators make the resulting objectives and constraints nonsmooth, which can severely hinder gradient-based optimization. This has motivated a large body of work on smooth approximations of DT-STL robustness. Log-sum-exp relaxations were introduced in~\cite{pant2017smooth} and extended to cumulative robustness in~\cite{haghighi2019control}; softmax-based approximations were proposed in~\cite{gilpin2020smooth}; and polynomial smoothings were used in~\cite{mao2022successive}. While these formulations improve differentiability, they are inherently approximate and typically recover exact Boolean consistency only in an asymptotic limit of the smoothing parameter.
A different line of work seeks to improve numerical behavior by reducing the dependence of robustness on isolated critical samples. This is motivated by two well-known pathologies of standard robustness, namely \emph{locality}, where the value depends on a single time instant, and \emph{masking}, where one subformula dominates a conjunction or disjunction and suppresses information from the others~\cite{mehdipour2019average}. Averaged STL~\cite{akazaki2015time}, discrete average space robustness (DASR) and simplified DASR~\cite{lindemann2019robust}, and arithmetic--geometric mean (AGM) robustness~\cite{mehdipour2019arithmetic} all attempt to distribute information more evenly across time and subformulas. These semantics improve certain numerical aspects, but they do not fully resolve the tension between smoothness and exactness. Some retain nonsmooth operators, some require auxiliary soundness constraints, and some apply only to restricted classes of formulas. Weighted generalizations of SR and AGM have also been developed to encode preferences among sub-specifications~\cite{cardona2023mixed,mehdipour2020specifying}. Another exact alternative is proposed in~\cite{varnai2020robustness}, but the resulting functions are non-monotone and not $\mathcal{C}^1$, which limits their suitability for numerical optimization.
More recently,~\cite{uzun2024optimization} introduced the generalized mean-based smooth robustness (D-GMSR), a $\mathcal{C}^1$-smooth and exact parameterization of DT-STL obtained by replacing the key $\min$ and $\max$ operators with generalized-mean compositions. Unlike prior smooth relaxations, D-GMSR is both sound and complete with respect to the underlying STL semantics. At the same time, its first-order sensitivity properties directly mitigate locality and masking, yielding a substantially better optimization landscape for gradient-based solvers. D-GMSR therefore provides a smooth, exact, and numerically attractive DT-STL parameterization.

\vspace{-0.2cm}
For continuous-time systems, however, enforcing STL only on a finite temporal grid is often inadequate. For always-type specifications, node-wise satisfaction does not preclude inter-sample violations and therefore cannot by itself guarantee safety along the full trajectory. For eventually-type specifications, restricting satisfaction to mesh points can be unnecessarily conservative and may degrade optimality, since the desired event may occur between nodes even when no sampled state satisfies the predicate. These issues motivate the study of continuous-time STL (CT-STL), where the semantics are enforced over dense time rather than only over a discretization grid.
A steadily growing literature addresses dense-time STL satisfaction for continuous-time systems. On the control-synthesis side, several works translate STL requirements into time-varying invariance and reachability conditions and enforce them using control barrier functions (CBFs) or related barrier-based controllers. Representative examples include foundational CBF constructions~\cite{lindemann2018control}, continuous-time model predictive control (MPC) formulations~\cite{Charitidou2021BFMPCSTL}, actuation-constrained CBF frameworks~\cite{Buyukkocak2022ActuationCBF}, and extensions to certain nested formulas~\cite{Yu2024NestedSTL}. On the trajectory-optimization side, dense-time guarantees have so far been obtained mainly through a few specialized mechanisms. One approach augments mixed-integer planners with CBF-derived intersample conditions under zero-order hold, thereby restoring dense-time correctness between sampling instants~\cite{Yang2020CTSTLCBF}. Another guarantees continuous-time satisfaction for quadrotor missions by planning over sparse waypoint sequences, enforcing strictified sampled-time specifications, and combining these with intersample deviation bounds and a hierarchical tracking architecture~\cite{pant2018fly}. More recent methods optimize directly over continuous trajectory families, including sound mixed-integer formulations over continuous B\'ezier trajectories for temporally robust multi-agent planning~\cite{Verhagen2024TemporallyRobust}, B\'ezier control-point and acceleration conditions for CT-STL satisfaction with time-varying robustness~\cite{Yuan2024TimeVaryingRobustness}, and their extension to multiple quadrotors through piecewise B\'ezier reference generation with explicit tracking-error bounds~\cite{Yuan2025Quadrotors}. A recent sampling-based planner also encodes a fragment of STL into a forward-invariant time-varying set and searches for dynamically feasible trajectories within that set~\cite{Marchesini2025SamplingBased}. On the semantics side, average-based robustness has been extended to continuous-time signals~\cite{mehdipour2019average}, but these formulations remain nonsmooth and do not by themselves yield a tractable reformulation for standard gradient-based nonlinear programming.

\vspace{-0.2cm}
In parallel, continuous-time successive convexification frameworks were developed to ensure continuous-time satisfaction of path constraints, i.e., \always{}-type specifications, in trajectory optimization~\cite{ctcs2024}. These methods augment the system dynamics with penalized path-constraint states, so that continuous-time enforcement of the original path constraints is achieved by enforcing satisfaction of the corresponding augmented dynamical constraints. Their numerical effectiveness has been demonstrated in several applications, including GPU-accelerated trajectory optimization for 6-DoF rocket landing~\cite{chari2024fast}, nonlinear model predictive control for obstacle avoidance~\cite{nmpc2024}, and trajectory optimization for 6-DoF aircraft approach and landing~\cite{aircraft2024}. This framework was later combined with GMSR in~\cite{uzun2025sequential} to handle the continuous-time satisfaction of logical \always{}-type specifications in a 6-DoF rocket landing problem, and its practical utility was further illustrated in perception-constrained drone motion planning~\cite{uzun2025motion}, in a related rocket-landing problem solved using an auto-tuned SCP method~\cite{mceowen2026auto}. In parallel, these ideas were incorporated into SCvxGEN, an automatic code-generation framework for fast embedded trajectory optimization~\cite{scvxgen}.
More recently, a dense-time GMSR-based CT-STL parameterization was proposed in~\cite{uzun2026smooth}. That approach evaluates temporal operators directly on the integration subnodes used for numerical dynamics propagation, yielding a smooth and exact parameterization of CT-STL specifications on the numerical trajectory representation, up to the accuracy of the integration scheme.
From an implementation perspective, however, dense-time semantic evaluation introduces an additional layer of derivative bookkeeping. The integration-subnode states and their sensitivities with respect to the shooting variables must be retained or reconstructed in order to assemble STL derivatives through the chain rule. This can increase memory requirements and implementation complexity, especially for long horizons, high-dimensional systems, or fine integration grids.
Taken together, these developments show that continuous-time satisfaction can be handled effectively through both augmented path-constraint dynamics and exact dense-time semantic evaluation. However, the existing literature still leaves open a complementary trajectory-optimization framework in which richer CT-STL temporal aggregation, including operators such as \eventually{} and \until{}, is embedded directly into augmented continuous-time dynamics and integrated with free-final-time optimal control, finite-dimensional control parameterization, multiple-shooting discretization, and convergence-guaranteed sequential convex programming.

\vspace{-0.2cm}
Motivated by this gap, this paper develops a smooth augmentation-based parameterization of continuous-time Signal Temporal Logic (CT-STL) specifications for nonlinear trajectory optimization. Building on the exact and smooth D-GMSR formulation~\cite{uzun2024optimization}, the proposed framework lifts the discrete-time aggregation structure of GMSR to continuous time by replacing finite sums and product-type aggregates with integral and geometric-integral counterparts over the relevant time intervals~\cite{stanley1999multiplicative, bashirov2008multiplicative}. These continuous-time aggregates are then embedded into the optimal-control problem through auxiliary dynamical states appended to the original system.
Once these auxiliary states are included in the dynamics, their sensitivities are handled by the same multiple-shooting and automatic-differentiation machinery used for the original states. This absorbs the temporal-logic bookkeeping into the optimal-control transcription and avoids the need to explicitly retain or reconstruct integration-subnode states and their sensitivities solely for assembling STL derivatives through the chain rule.
The resulting framework preserves the continuous-time nature of temporal specifications within the trajectory-optimization problem while retaining the favorable smoothness and gradient properties of GMSR. In particular, it provides differentiable representations of \always{}, \eventually{}, and \until{} specifications over prescribed time intervals, mitigates the locality and masking issues of standard quantitative semantics, and remains largely transcription-independent. As a result, the proposed augmentation-based realization can be incorporated into existing optimal-control pipelines with minimal structural changes.
The numerical performance of the proposed framework is demonstrated through trajectory-optimization examples involving a double-integrator system under separate \always{}, \eventually{}, and \until{} specifications, as well as a 6-DoF quadrotor flight problem with combined \always{}, \implication{}, and \eventually{}-type specifications.

\vspace{-0.2cm}
\section{Preliminaries}\label{sec:stl}
\vspace{-0.2cm}
This section reviews the standard syntax and semantics of Signal Temporal Logic (STL) for continuous-time signals and fixes the notation used throughout the paper. Let $x:\mathbb{R}_{+}\to\mathbb{R}^{n}$ be a continuous-time signal. A Boolean-valued atomic predicate $\mu$ is induced by a continuously differentiable predicate function $g:\mathbb{R}^{n}\to\mathbb{R}$ according to
$
    \mu := (g(x(t)) \ge 0).
$
Thus, the predicate $\mu$ is \true{} at time $t$ if and only if $g(x(t))\ge 0$.

\vspace{-0.2cm}
An STL formula $\varphi$ is defined recursively as
\begin{equation*}\label{eq:stl_syntax}
    \varphi ::= \top \mid \mu \mid \neg \varphi \mid \varphi_1 \wedge \varphi_2 \mid \varphi_1 \bm{U}_{[a,b]} \varphi_2,
\end{equation*}
where $\varphi_1$ and $\varphi_2$ are STL subformulas, $\neg$ and $\wedge$ denote the Boolean operators of \negation{} and \stland{} (\conjunction{}), respectively, and $\bm{U}_{[a,b]}$ denotes the bounded \until{} operator over the interval $[a,b]\subset\mathbb{R}_{+}$.

\vspace{-0.2cm}
The standard derived operators \stlor{} (\disjunction{}), \implication{}, \eventually{}, and \always{} are recovered from the fundamental operators via
$
\varphi_1 \vee \varphi_2 \equiv \neg(\neg \varphi_1 \wedge \neg \varphi_2),
\,
\varphi_1 \implies \varphi_2 \equiv \neg \varphi_1 \vee \varphi_2,
$
and
$
\bm{F}_{[a,b]}\varphi \equiv \top \bm{U}_{[a,b]} \varphi,
\,
\bm{G}_{[a,b]}\varphi \equiv \neg \bm{F}_{[a,b]} \neg \varphi.
$

The Boolean semantics determine whether a signal satisfies a formula at a given time instant. We write $(x,t)\models\varphi$ to indicate that the signal $x$ satisfies $\varphi$ at time $t$.

\vspace{-0.2cm}
\begin{definition}[Continuous-time STL semantics \cite{donze2010robust}]
For a continuous-time signal $x:\mathbb{R}_{+}\to\mathbb{R}^{n}$ and time $t\in\mathbb{R}_{+}$, the Boolean satisfaction relation $(x,t)\models\varphi$ is defined recursively as summarized in Table~\ref{tab:stl_semantics}.
\end{definition}

\vspace{-0.2cm}
\begin{table}[!t]
\centering
\renewcommand{\arraystretch}{1.25}
\begin{tabular}{ll}
\toprule
\textbf{Formula} $\varphi$ & \textbf{Boolean semantics} $(x,t)\models\varphi$ \\
\midrule
$\mu$ & $g(x(t)) \ge 0$ \\
$\neg \varphi$ & $\neg\big((x,t)\models \varphi\big)$ \\
$\varphi_1 \wedge \varphi_2$ & $(x,t)\models \varphi_1 \;\wedge\; (x,t)\models \varphi_2$ \\
$\varphi_1 \vee \varphi_2$ & $(x,t)\models \varphi_1 \;\vee\; (x,t)\models \varphi_2$ \\
$\varphi_1 \implies \varphi_2$ & $\neg\big((x,t)\models \varphi_1\big)\;\vee\; (x,t)\models \varphi_2$ \\
\midrule
$\bm{G}_{[a,b]}\varphi$ & $\forall\, t' \in [t+a,t+b],\; (x,t')\models \varphi$ \\
$\bm{F}_{[a,b]}\varphi$ & $\exists\, t' \in [t+a,t+b]\ \text{s.t.}\ (x,t')\models \varphi$ \\
$\varphi_1 \bm{U}_{[a,b]} \varphi_2$
& $\exists\, t_1 \in [t+a,t+b]\ \text{s.t.}\ (x,t_1)\models \varphi_2$ \\
& \quad and $\forall\, t_2 \in [t,t_1],\; (x,t_2)\models \varphi_1$ \\
\bottomrule
\end{tabular}
\caption{Continuous-time Boolean semantics for STL.}
\label{tab:stl_semantics}
\end{table}

\vspace{-0.1cm}
While Boolean semantics provide a binary \true{}/\false{} evaluation, optimal control requires a continuous-valued measure of satisfaction that can be used within numerical solvers. A standard choice is the quantitative robustness semantics introduced in Table~\ref{tab:stl_robustness}~\cite{donze2010robust}. Let $\rho^\varphi(x,t)$ denote the robustness of formula $\varphi$ at time $t$. These semantics are sign-consistent with Boolean satisfaction: $\rho^\varphi(x,t)>0$ implies strict satisfaction, $\rho^\varphi(x,t)<0$ implies violation, and $\rho^\varphi(x,t)=0$ corresponds to the satisfaction boundary.

\vspace{-0.2cm}
\begin{table}[!t]
\centering
\renewcommand{\arraystretch}{1.25}
\begin{tabular}{ll}
\toprule
\textbf{Formula} $\varphi$ & \textbf{Quantitative robustness} $\rho^\varphi(x,t)$ \\
\midrule
$\mu$ & $g(x(t))$ \\
$\neg \varphi$ & $-\rho^\varphi(x,t)$ \\
$\varphi_1 \wedge \varphi_2$ & $\min\!\big(\rho^{\varphi_1}(x,t),\,\rho^{\varphi_2}(x,t)\big)$ \\
$\varphi_1 \vee \varphi_2$ & $\max\!\big(\rho^{\varphi_1}(x,t),\,\rho^{\varphi_2}(x,t)\big)$ \\
$\varphi_1 \implies \varphi_2$ & $\max\!\big(-\rho^{\varphi_1}(x,t),\,\rho^{\varphi_2}(x,t)\big)$ \\
\midrule
$\bm{G}_{[a,b]}\varphi$ & $\min\limits_{t' \in [t+a,t+b]} \rho^\varphi(x,t')$ \\
$\bm{F}_{[a,b]}\varphi$ & $\max\limits_{t' \in [t+a,t+b]} \rho^\varphi(x,t')$ \\
$\varphi_1 \bm{U}_{[a,b]} \varphi_2$
& $\max\limits_{t_1 \in [t+a,t+b]}
\min\!\Big(\rho^{\varphi_2}(x,t_1),\,
\min\limits_{t_2 \in [t,t_1]} \rho^{\varphi_1}(x,t_2)\Big)$ \\
\bottomrule
\end{tabular}
\caption{Standard quantitative robustness semantics for continuous-time STL.}
\label{tab:stl_robustness}
\end{table}

\vspace{-0.1cm}
Although the robustness semantics in Table~\ref{tab:stl_robustness} preserve the truth conditions of STL, they rely on nested $\min/\max$ operators and are therefore generally nonsmooth. This nonsmooth structure can make them difficult to embed directly in gradient-based optimal control methods, especially when temporal operators must be enforced over dense time. In the developments that follow, these standard semantics serve as the reference notion of correctness, while the proposed formulation introduces a smooth and optimization-friendly representation that remains faithful to continuous-time STL satisfaction.
\section{Successive Convexification for Trajectory Optimization with CT-STL Specifications}
\vspace{-0.2cm}
We consider a free-final-time optimal control problem in Mayer form, subject to continuous-time Signal Temporal Logic (CT-STL) specifications. The problem is formulated as follows:
\begin{subequations} \label{ct-ocp}
    \begin{align}
        \underset{x, \, u, \, t_{\mathrm{f}}}{\mathrm{minimize}} \quad & L\big(t_{\mathrm{f}}, x(t_{\mathrm{f}})\big) \label{ct-cost} \\
        \mathrm{subject~to} \quad & \dot{x}(t) = F\big(t, x(t), u(t)\big), \; \mathrm{a.e.} \; t \in [0, t_{\mathrm{f}}], \\
        & P\big( x(0), u(0), x(t_{\mathrm{f}}), u(t_{\mathrm{f}}) \big) \leq 0, \\
        & Q\big( x(0), u(0), x(t_{\mathrm{f}}), u(t_{\mathrm{f}}) \big) = 0, \\
        & (x, 0) \models \varphi, \label{eq:ct-stl-constraint}
    \end{align}
\end{subequations}
where $L$ represents the terminal cost, $F$ governs the nonlinear system dynamics, $P$ and $Q$ denote the inequality and equality boundary conditions respectively, and $\varphi$ is the global STL specification evaluated at the initial time $t=0$.

\vspace{-0.2cm}
\begin{rem} \label{rem:ctcs-always}
Standard continuous-time path constraints $g(x(t)) \leq 0$ and equality constraints $h(x(t)) = 0$ enforced over the entire trajectory $t \in [0, t_{\mathrm f}]$ do not require separate mathematical treatment in this framework. They can be expressed equivalently through the STL satisfaction condition
\begin{align*}
    (x,0)\models
    \bm{G}_{[0,t_{\mathrm f}]}
    \!
    \big(
        (-g(x)\!\ge\!0)
        \! \wedge \!
        (h(x)\!\ge\!0)
        \! \wedge \!
        (-h(x)\!\ge\!0)
    \big).
\end{align*}
This allows all state-dependent constraints to be unified and evaluated under a single, smooth robustness parameterization.
\end{rem}

\vspace{-0.2cm}
The remainder of this section develops the complete successive-convexification pipeline for solving~\eqref{ct-ocp}. We first construct a GMSR-based parameterization of CT-STL specifications by lifting the finite-dimensional aggregation structure of D-GMSR to continuous time and realizing the resulting temporal aggregates through auxiliary dynamical states. We then incorporate time-dilation for free-final-time problems, introduce a finite-dimensional control parameterization, discretize the augmented dynamics using multiple shooting, and obtain a finite-dimensional nonlinear optimal control problem. Finally, we apply the prox-convex sequential convex programming method to solve the resulting nonconvex program.
The main body focuses on this augmentation-based realization because it embeds temporal aggregation directly into the optimal-control dynamics and allows the associated sensitivities to be handled by the same multiple-shooting and automatic-differentiation machinery used for the original system. For completeness, Appendix~\ref{ssec:exact_dense_ctstl} summarizes the exact dense-time realization of~\cite{uzun2026smooth}, which evaluates GMSR temporal operators on integration subnodes, and briefly discusses its extension to time-dilated trajectories.

\vspace{-0.2cm}
\subsection{Parameterization of the CT-STL specifications}
\vspace{-0.2cm}
To incorporate CT-STL specifications into trajectory optimization, we use the generalized mean-based smooth robustness (GMSR) construction introduced in \cite{uzun2024optimization}. Standard quantitative semantics typically face a fundamental trade-off. They are either exact but nonsmooth, due to repeated use of $\min$ and $\max$, or smooth but approximate. Moreover, common smooth approximations often suffer from \emph{locality}, where sensitivity is concentrated at a single critical time, and \emph{masking}, where satisfied subformulas suppress gradient information from violated ones. GMSR avoids these issues by providing a smooth and exact parameterization of the logical operators, with gradient information distributed across multiple relevant subformulas and time instants rather than collapsing onto a single critical sample.

\vspace{-0.2cm}
\subsubsection{Logical operators}
\vspace{-0.2cm}
For the logical operators \stland{} ($\wedge$) and \stlor{} ($\vee$), we use the simplified GMSR parameterization
\vspace{-0.2cm}
\begin{align*}
    {}^{\wedge} h^{c}(y)
    &\defeq
    \left(M_{0}^{c}\!\left(\relu{y}^{2}\right)\right)^{\frac12}
    -
    \left(M_{1}^{c}\!\left(\negrelu{y}^{2}\right)\right)^{\frac12}, \\
    {}^{\vee} h^{c}(y)
    &\defeq
    -{}^{\wedge} h^{c}(-y),
\end{align*}
where $y\in\mathbb R^n$, $\relu{y}:=\max(0,y)$, and $\negrelu{y}:=\min(0,y)$ are applied elementwise, and
\vspace{-0.2cm}
\begin{align*}
    M_{0}^{c}(z)
    \defeq
    \left(c^n+\prod_{i=1}^{n} z_i\right)^{1/n}, 
    \quad
    M_{1}^{c}(z)
    \defeq
    c+\frac{1}{n}\sum_{i=1}^{n} z_i,
\end{align*}
with $c\in\mathbb R_{++}$.

\vspace{-0.2cm}
\begin{rem}[Smoothing parameter and extensions]
The parameter $c$ acts as a positive regularization shift near the switching region. For any $c>0$, the resulting functions ${}^{\wedge} h^{c}$ and ${}^{\vee} h^{c}$ are $\mathcal C^1$-smooth with bounded gradients, while preserving the exact sign-based parameterization of conjunction and disjunction. Smaller values of $c$ produce a less flattened transition near zero, whereas larger values yield a smoother but flatter local landscape.

\vspace{-0.2cm}
The simplified form above corresponds to a uniform-weight, fixed-curvature choice within the more general GMSR family introduced in \cite{uzun2024optimization}. In that formulation, weight parameters can be used to emphasize selected inputs or time samples, while curvature parameters can make the aggregation more or less extremum-like. For clarity of presentation, we use the simplified form throughout this section. Further intuition, derivative properties, and plots for different choices of these parameters are given in \cite{uzun2024optimization}.
\end{rem}

\vspace{-0.2cm}
We now define the robustness measure recursively. For an atomic predicate
$
\mu := (g(x)\ge 0),
$
let
\vspace{-0.2cm}
\begin{align*}
    \Gamma^{\mu}(x,t) := g(x(t)).
\end{align*}
\Negation{}, \conjunction{}, \disjunction{}, and \implication{} are then parameterized as
\vspace{-0.2cm}
\begin{align*}
    \Gamma_{\bm c}^{\neg\varphi}(x,t)
    &:= -\Gamma_{\bm c}^{\varphi}(x,t), \\
    \Gamma_{\bm c}^{\varphi_1\wedge\varphi_2}(x,t)
    &:= {}^{\wedge} h^{c_1}
    \Big(
        \Gamma_{\bm c}^{\varphi_1}(x,t),\,
        \Gamma_{\bm c}^{\varphi_2}(x,t)
    \Big), \\
    \Gamma_{\bm c}^{\varphi_1\vee\varphi_2}(x,t)
    &:= {}^{\vee} h^{c_1}
    \Big(
        \Gamma_{\bm c}^{\varphi_1}(x,t),\,
        \Gamma_{\bm c}^{\varphi_2}(x,t)
    \Big), \\
    \Gamma_{\bm c}^{\varphi_1\Rightarrow\varphi_2}(x,t)
    &:= {}^{\vee} h^{c_1}
    \Big(
        -\Gamma_{\bm c}^{\varphi_1}(x,t),\,
        \Gamma_{\bm c}^{\varphi_2}(x,t)
    \Big),
\end{align*}
where $\bm c$ collects the positive shift parameters appearing in the recursive GMSR compositions.

\vspace{-0.2cm}
Accordingly,
\vspace{-0.2cm}
\begin{align*}
    (x,t)\models \varphi
    \iff
    \Gamma_{\bm c}^{\varphi}(x,t)\ge 0
\end{align*}
for any formula $\varphi$ built recursively from the above logical operators.

\vspace{-0.2cm}
\begin{rem}
Although the conditions
\vspace{-0.2cm}
\begin{align*}
    {}^{\wedge} h^{c}(y)\ge 0
    &\iff
    \sum_{i=1}^{n}\negrelu{y_i}^{2}=0, \\
    {}^{\vee} h^{c}(y)\ge 0
    &\iff
    \prod_{i=1}^{n}\negrelu{y_i}^{2}=0
\end{align*}
are logically equivalent, the full GMSR expressions are numerically preferable. The geometric-mean normalization, together with the outer square-root, keeps the robustness values and gradients better scaled, while the equivalent sum and product tests can be poorly conditioned. Moreover, the shift parameter $c>0$ regularizes the switching region, ensures $\mathcal {C} ^1$-smoothness, and avoids singular behavior near zero.
\end{rem}

\vspace{-0.2cm}
\subsubsection{Temporal operators}
\vspace{-0.2cm}
A continuous-time realization of CT-STL temporal operators can be obtained through dynamical augmentation. This realization is transcription-independent and can often be incorporated into an existing optimal-control formulation simply by augmenting the state dynamics. It is based on replacing the discrete multiplication and summation appearing in the GMSR operators by their continuous-time analogues, namely a normalized geometric integral \cite{stanley1999multiplicative, bashirov2008multiplicative} and a normalized ordinary integral.

\vspace{-0.2cm}
For a scalar signal $y(t)$ on an interval $I=[a,b]$, define
\vspace{-0.2cm}
\begin{align*}
    \mathcal M_{0,\varepsilon}^{c}(y;I)
    &\defeq
    c+\exp\!\left(
        \frac{1}{|I|}
        \int_a^b \log\!\big(\relu{y(t)}^2+\varepsilon\big)\,\mathrm dt
    \right), \\
    \mathcal M_{1}^{c}(y;I)
    &\defeq
    c+\frac{1}{|I|}
    \int_a^b \negrelu{y(t)}^2\,\mathrm dt.
\end{align*}
The corresponding temporal surrogates are
\vspace{-0.2cm}
\begin{align*}
    {}^{\wedge}\mathcal H_{\varepsilon}^{c}(y;I)
    &\defeq
    \big(\mathcal M_{0,\varepsilon}^{c}(y;I)\big)^{\frac12}
    -
    \big(\mathcal M_{1}^{c}(y;I)\big)^{\frac12}, \\
    {}^{\vee}\mathcal H_{\varepsilon}^{c}(y;I)
    &\defeq
    -{}^{\wedge}\mathcal H_{\varepsilon}^{c}(-y;I).
\end{align*}
The shift $\varepsilon>0$ is required because $\relu{y(t)}^2$ may vanish, in which case the logarithm is undefined. Consequently, the augmentation-based realization is only approximately exact for any fixed $\varepsilon>0$; mitigation of the resulting approximation effects is discussed after the operator constructions below.

\vspace{-0.2cm}
These temporal aggregates can be implemented by augmenting the dynamics. In this paper, we use a multiplicative realization for the geometric term by introducing the auxiliary state
\vspace{-0.2cm}
\begin{align*}
    \dot \eta(t)
    =
    \frac{1}{|I|}\,\eta(t)\log\!\big(\relu{y(t)}^2+\varepsilon\big),
    \qquad
    \eta(a)=1,
\end{align*}
\vspace{-0.2cm}
so that
\vspace{-0.2cm}
\begin{align*}
    \eta(b)=\exp\!\left(
        \frac{1}{|I|}
        \int_a^b \log\!\big(\relu{y(t)}^2+\varepsilon\big)\,\mathrm dt
    \right),
\end{align*}
\vspace{-0.2cm}
and therefore
\vspace{-0.2cm}
\begin{align*}
    \mathcal M_{0,\varepsilon}^{c}(y;I)=c+\eta(b).
\end{align*}
Because the dynamics are normalized by $|I|$, the state evolves on the scale of the geometric mean rather than the geometric product. In particular, if $\relu{y(t)}^2=0$ on $I$, then $\eta(b)=\varepsilon$.

\vspace{-0.2cm}
The additive realization of the integral term is the same as that used in \cite{ctcs2024} for continuous-time satisfaction of path constraints
\vspace{-0.2cm}
\begin{align*}
    \dot \xi(t)
    =
    \frac{1}{|I|}\negrelu{y(t)}^2,
    \qquad
    \xi(a)=0,
\end{align*}
\vspace{-0.2cm}
so that
\vspace{-0.2cm}
\begin{align*}
    \mathcal M_{1}^{c}(y;I)=c+\xi(b).
\end{align*}

\vspace{-0.2cm}
A logarithmic realization of the geometric term is also possible, obtained by introducing
\vspace{-0.2cm}
\begin{align*}
    \dot \zeta(t)
    =
    \frac{1}{|I|}\log\!\big(\relu{y(t)}^2+\varepsilon\big),
    \qquad
    \zeta(a)=0,
\end{align*}
in which case $\eta(b)=e^{\zeta(b)}$. Both realizations are equivalent in exact arithmetic. In what follows, however, we use the multiplicative realization.

\vspace{-0.2cm}
For an interval $I=[a,b]$, let $\mathbf 1_I(t)$ denote the indicator function of $I$, namely $\mathbf 1_I(t)=1$ if $t\in I$ and $\mathbf 1_I(t)=0$ otherwise.

\vspace{-0.2cm}
We now introduce operator-specific augmented dynamics. For each temporal formula $\psi$, we define an augmented state $x^\psi(t)$, and vector field $F^\psi$.

\vspace{-0.2cm}
For \always{}, $\psi:=\bm G_I\varphi$ with $I=[a,b]$, define the augmented state
\vspace{-0.2cm}
\begin{align*}
    x^{\psi}(t)
    :=
    \big(
        x(t),
        \eta_G(t),
        \xi_G(t)
    \big),
\end{align*}
with augmented dynamics
\vspace{-0.2cm}
\begin{align*}
    \dot{\eta}_G(t)
    &=
    \frac{1}{|I|}
    \mathbf 1_I(t)\,
    \eta_G(t)\,
    \log \bigl(|\Gamma_{\bm c}^{\varphi}(x,t)|_+^2+\varepsilon\bigr),
    \\
    \dot{\xi}_G(t)
    &=
    \frac{1}{|I|}
    \mathbf 1_I(t)\,
    |\Gamma_{\bm c}^{\varphi}(x,t)|_-^2,
\end{align*}
and initial auxiliary conditions
$
    \eta_G(0)=1,
    \,
    \xi_G(0)=0.
$
The corresponding parameterized robustness is
\vspace{-0.2cm}
\begin{align*}
    \Gamma_{\bm c,\varepsilon}^{\bm G_I\varphi}(x,0)
    &:=
    \big(c+\eta_G(t_{\mathrm f})\big)^{\frac12}
    -
    \big(c+\xi_G(t_{\mathrm f})\big)^{\frac12}.
\end{align*}

\vspace{-0.2cm}
For \eventually{}, $\psi:=\bm F_I\varphi$ with $I=[a,b]$, define the augmented state
\vspace{-0.2cm}
\begin{align*}
    x^{\psi}(t)
    :=
   \big(
        x(t),
        \eta_F(t),
        \xi_F(t)
    \big),
\end{align*}
with augmented dynamics
\vspace{-0.2cm}
\begin{align*}
    \dot{\eta}_F(t)
    &=
    \frac{1}{|I|}
    \mathbf 1_I(t)\,
    \eta_F(t)\,
    \log \bigl(|\Gamma_{\bm c}^{\varphi}(x,t)|_-^2+\varepsilon\bigr),
    \\
    \dot{\xi}_F(t)
    &=
    \frac{1}{|I|}
    \mathbf 1_I(t)\,
    |\Gamma_{\bm c}^{\varphi}(x,t)|_+^2,
\end{align*}
and initial auxiliary conditions
$
    \eta_F(0)=1,
    \,
    \xi_F(0)=0.
$
The corresponding parameterized robustness is
\vspace{-0.2cm}
\begin{align*}
    \Gamma_{\bm c,\varepsilon}^{\bm F_I\varphi}(x,0)
    &:=
    \big(c+\xi_F(t_{\mathrm f})\big)^{\frac12}
    -
    \big(c+\eta_F(t_{\mathrm f})\big)^{\frac12}.
\end{align*}

\vspace{-0.2cm}
For \until{}, $\psi:=\varphi_1 \bm U_I \varphi_2$ with $I=[a,b]$, let
\vspace{-0.2cm}
\[
\Gamma_{\bm c}^{\varphi_1}(x,t)
\quad \text{and} \quad
\Gamma_{\bm c}^{\varphi_2}(x,t)
\]
denote the robustness signals of the two subformulas. We first construct a running prefix surrogate for the statement that $\varphi_1$ has held on $[0,t]$. Introduce
\vspace{-0.2cm}
\begin{align*}
    \dot{\eta}_1(t)
    &=
    \eta_1(t)\log\!\big(\relu{\Gamma_{\bm c}^{\varphi_1}(x,t)}^2+\varepsilon\big),
    \qquad
    \eta_1(0)=1, \\
    \dot{\xi}_1(t)
    &=
    \negrelu{\Gamma_{\bm c}^{\varphi_1}(x,t)}^2,
    \qquad
    \xi_1(0)=0.
\end{align*}
Define the prefix-normalized quantities
\vspace{-0.2cm}
\begin{align*}
    \hat\eta_1(t)
    &:=
    \begin{cases}
        \eta_1(t)^{1/t}, & t>0,\\
        \relu{\Gamma_{\bm c}^{\varphi_1}(x,0)}^2+\varepsilon, & t=0,
    \end{cases} \\
    \hat\xi_1(t)
    &:=
    \begin{cases}
        \xi_1(t)/t, & t>0,\\
        \dot{\xi}_1(0), & t=0,
    \end{cases}
\end{align*}
and then set
\vspace{-0.2cm}
\begin{align*}
    q_1(t)
    :=
    \Big(c_3+\hat\eta_1(t)\Big)^{\frac12}
    -
    \Big(c_3+\hat\xi_1(t)\Big)^{\frac12}.
\end{align*}
This gives a continuous surrogate of the statement that $\varphi_1$ has held on the prefix interval $[0,t]$.

\vspace{-0.2cm}
\begin{rem}
The prefix normalization by $t$ in $\hat\eta_1(t)$ and $\hat\xi_1(t)$ may be regularized in practice by replacing $t$ with $t+\delta$, where $\delta>0$ is small. This avoids the singularity at $t=0$ and can improve numerical behavior. The resulting construction is no longer exactly identical to the ideal prefix normalization, but the perturbation can be made arbitrarily small as $\delta\to 0^+$.
\end{rem}

\vspace{-0.2cm}
Next, combine this running prefix signal with $\varphi_2$ at each time using the logical GMSR conjunction
\vspace{-0.2cm}
\begin{align*}
    z(t)
    :=
    {}^{\wedge}h^{c_2} \big(\Gamma_{\bm c}^{\varphi_2}(x,t),q_1(t)\big).
\end{align*}
Collecting the original state, the prefix states, and the outer witness states, define the augmented state
\vspace{-0.2cm}
\begin{align*}
    x^{\psi}(t)
    :=
    \big(
        x(t),
        \eta_1(t),
        \xi_1(t),
        \eta_U(t),
        \xi_U(t)
    \big),
\end{align*}
with augmented dynamics for outer witness states
\vspace{-0.2cm}
\begin{align*}
    \dot \eta_U(t) &= \dfrac{1}{|I|}\,\mathbf 1_I(t)\,
    \eta_U(t)\log\!\big(\negrelu{z(t)}^2+\varepsilon\big)\\
    \dot \xi_U(t) &= \dfrac{1}{|I|}\,\mathbf 1_I(t)\,
        \relu{z(t)}^2
\end{align*}
and initial auxiliary conditions
$
    \eta_U(0)=1
$,
$
    \xi_U(0)=0.
$
The corresponding parameterized robustness is
\vspace{-0.2cm}
\begin{align*}
    \Gamma_{\bm c,\varepsilon}^{\varphi_1 \bm U_I \varphi_2}(x,0)
    &:=
    \big(c_1+\xi_U(t_{\mathrm f})\big)^{\frac12}
    -
    \big(c_1+\eta_U(t_{\mathrm f})\big)^{\frac12}.
\end{align*}

\vspace{-0.25cm}
It is convenient to absorb both the auxiliary initial conditions and the terminal robustness inequality into augmented boundary maps. Accordingly, we define
\vspace{-0.2cm}
\begin{align*}
    P^\psi\big(x^\psi(0),u(0),x^\psi(t_{\mathrm f}),u(t_{\mathrm f})\big)
    &:=
    \begin{bmatrix}
        P\big(z_0,z_{\mathrm f}\big)\\[0.1cm]
        -\Gamma_{\bm c,\varepsilon}^{\psi}(x,0)
    \end{bmatrix},\\
    Q^\psi\big(x^\psi(0),u(0),x^\psi(t_{\mathrm f}),u(t_{\mathrm f})\big)
    &:=
    \begin{bmatrix}
        Q\big(z_0,z_{\mathrm f}\big)\\[0.1cm]
        x^\psi_{\mathrm{aux}}(0)-x^\psi_{\mathrm{aux},0}
    \end{bmatrix},
\end{align*}
where $z_0 := \big(x(0),u(0)\big)$, $z_{\mathrm f} := \big(x(t_{\mathrm f}),u(t_{\mathrm f})\big)$. Here, $x^\psi_{\mathrm{aux}}(t)$ denotes the auxiliary block of $x^\psi(t)$, and $x^\psi_{\mathrm{aux},0}$ collects the corresponding prescribed initial values. With this notation, the CT-STL-augmented continuous-time optimal control problem takes the form
\vspace{-0.2cm}
\begin{subequations}\label{eq:ct_ocp_aug}
\begin{align*}
    \underset{x^\psi,\,u,\,t_{\mathrm f}}{\mathrm{minimize}}
    \quad
    &L\big(t_{\mathrm f},x(t_{\mathrm f})\big) \\
    \mathrm{subject\ to}\quad
    &\dot x^\psi(t)=F^\psi\big(t,x^\psi(t),u(t)\big),
    \quad \mathrm{a.e.}\ t\in[0,t_{\mathrm f}], \\
    &P^\psi\big(x^\psi(0),u(0),x^\psi(t_{\mathrm f}),u(t_{\mathrm f})\big)\le 0, \\
    &Q^\psi\big(x^\psi(0),u(0),x^\psi(t_{\mathrm f}),u(t_{\mathrm f})\big)=0.
\end{align*}
\end{subequations}

\vspace{-0.2cm}
\begin{rem}
Although the augmentation-based realization is not strictly exact, the practical effect of the $\varepsilon$-shift can often be made arbitrarily small by combining asymmetric weighting with small physical buffers. For critical \always{} constraints such as obstacle avoidance, a small buffer in the predicate restores strict engineering safety, while increasing the relative weight of the integral term reduces the influence of the shifted geometric term. As this weighting becomes more skewed toward the integral term, the required buffer can be made arbitrarily small. In the limiting case, one may remove the geometric-integral term entirely and retain only the integral penalty on the negative part, at the cost of losing the positive-margin shaping effect of GMSR. For \eventually{} specifications, if the target set has a nonempty interior, the integral term over the positive margin can dominate the shifted geometric term, making the induced conservatism arbitrarily small. The same ideas carry over to \until{} through its inner \always{}-type prefix condition and outer \eventually{}-type witness condition.
\end{rem}

\vspace{-0.2cm}
In summary, the augmentation-based realization embeds CT-STL temporal aggregation directly into auxiliary continuous-time dynamics, facilitating its integration with multiple shooting and automatic differentiation without explicitly retaining dense-time samples and their sensitivities solely for STL derivative assembly. The tradeoff for this transcription-friendly structure is the additional $\varepsilon$-shift in the geometric-integral term, whose practical effect can be made small through the mitigation strategies discussed above. This augmentation-based realization is the main CT-STL construction used in the finite-dimensional trajectory-optimization framework below. An exact dense-time alternative that avoids this $\varepsilon$-shift by evaluating GMSR temporal operators on integration subnodes is summarized in Appendix~\ref{ssec:exact_dense_ctstl}.

\vspace{-0.2cm}
\subsubsection{Weighted extensions}
\vspace{-0.2cm}
The simplified formulas above correspond to the basic choice of uniform weighting and fixed curvature in the original GMSR family. More generally, the logical operators can be written in the weighted form
\vspace{-0.2cm}
\begin{align*}
    {}^{\wedge} h_{p,w}^{c}(y)
    &\defeq
    \left( M_{0,w}^{c}\!\left(\relu{y}^{2}\right) \right)^{\frac12}
    -
    \left( M_{p,w}^{c}\!\left(\negrelu{y}^{2}\right) \right)^{\frac12}, \\
    {}^{\vee} h_{p,w}^{c}(y)
    &\defeq
    -{}^{\wedge} h_{p,w}^{c}(-y),
\end{align*}
where
\vspace{-0.2cm}
\begin{align*}
    M_{0,w}^{c}(z)
    &\defeq
    \left(c^{\mathbf 1^T w}+\prod_{i=1}^n z_i^{w_i}\right)^{1/(\mathbf 1^T w)}, \\
    M_{p,w}^{c}(z)
    &\defeq
    \left(c^p+\frac{1}{\mathbf 1^T w}\sum_{i=1}^n w_i z_i^p\right)^{1/p}.
\end{align*}
Here, $c>0$ is the smoothness shift, $p\in\mathbb Z_{++}$ controls the curvature of the power-mean term, and $w\in\mathbb Z_{++}^n$ encodes the relative importance of the inputs. Larger values of $p$ make the power-mean term closer to an extremum-type aggregation, while $p=1$ yields the arithmetic-mean choice used throughout this paper. Larger weights amplify the influence of the corresponding arguments on both the robustness value and its gradient.

\vspace{-0.2cm}
In the augmentation-based realization, the discrete weight vector is replaced by a nonnegative weighting density $\omega:I\to\mathbb R_+$ over the interval $I$, leading to weighted temporal aggregates of the form
\vspace{-0.2cm}
\begin{align*}
    \mathcal{M}^{c}_{0,\omega,\varepsilon}(y;I)
    &\defeq
    c + \exp\!\left(
        \frac{1}{\Omega_I}
        \int_I \omega(t)\log\!\big(\relu{y(t)}^2+\varepsilon\big)\,\mathrm dt
    \right), \\
    \mathcal{M}^{c}_{p,\omega}(y;I)
    &\defeq
    \left(
        c^p+
        \frac{1}{\Omega_I}
        \int_I \omega(t)\negrelu{y(t)}^{2p}\,\mathrm dt
    \right)^{1/p},
\end{align*}
where
\vspace{-0.2cm}
\[
\Omega_I := \int_I \omega(t)\,\mathrm dt.
\]
Accordingly, the weighted continuous-time temporal surrogate becomes
\vspace{-0.2cm}
\begin{align*}
    {}^{\wedge} \mathcal H^{c}_{p,\omega,\varepsilon}(y;I)
    &:=
    \left(\mathcal{M}^{c}_{0,\omega,\varepsilon}(y;I)\right)^{\frac12}
    -
    \left(\mathcal{M}^{c}_{p,\omega}(y;I)\right)^{\frac12}, \\
    {}^{\vee} \mathcal H^{c}_{p,\omega,\varepsilon}(y;I)
    &:=
    -{}^{\wedge} \mathcal H^{c}_{p,\omega,\varepsilon}(-y;I).
\end{align*}
Thus, the original GMSR parameters naturally extend to the continuous-time setting. The parameter $c$ controls regularity near switching, $p$ controls the curvature of the averaging term, and $w$ or $\omega$ controls the relative importance of subformulas or time instants.

\vspace{-0.2cm}
\subsubsection{Gradient behavior and optimization relevance}
\vspace{-0.2cm}
The favorable gradient properties of GMSR established in \cite{uzun2024optimization} carry over to the continuous-time constructions used here. The key point is that GMSR distributes derivative information across multiple relevant subformulas and time instants, rather than collapsing it onto a single critical sample as in standard $\min$/$\max$-based robustness. This is precisely why locality and masking are mitigated.

\vspace{-0.2cm}
At the level of the logical operators, the derivative of ${}^{\wedge}h_{p,w}^{c}$ or ${}^{\vee}h_{p,w}^{c}$ with respect to each active argument is shaped by the same weighting and curvature rules as in \cite{uzun2024optimization}. In the geometric-mean branch, the sensitivity is inversely proportional to the magnitude of the active argument, whereas in the power-mean branch it is proportional to $|y_i|^{2p-1}$. Consequently, the geometric branch emphasizes the most critical active margins, while the power-mean branch emphasizes the largest active deviations.

\vspace{-0.2cm}
The augmentation-based realization has the same qualitative structure, but in continuous form and with the $\varepsilon$-shift in the geometric term. If $y(t)$ is the active signal and $\omega(t)$ is the temporal weight density, then the conjunction-type surrogate ${}^{\wedge}\mathcal H_{p,\omega,\varepsilon}^{c}$ retains the same smooth worst-case interpretation:
\vspace{-0.2cm}
\begin{align*}
    \frac{\delta\,{}^{\wedge}\mathcal H_{p,\omega,\varepsilon}^{c}}{\delta y(t)}
    &\propto
    \frac{\omega(t) y(t)}{y(t)^2+\varepsilon}
    \quad \text{geometric-integral branch}, \\
    \frac{\delta\,{}^{\wedge}\mathcal H_{p,\omega,\varepsilon}^{c}}{\delta y(t)}
    &\propto
    \omega(t)|y(t)|^{2p-1}
    \quad \text{integral branch}.
\end{align*}
Thus, smaller positive margins and larger violations receive stronger emphasis across the interval. For the disjunction-type surrogate ${}^{\vee}\mathcal H_{p,\omega,\varepsilon}^{c}$, the interpretation is again reversed. The geometric-integral branch acts in a smooth best-case manner, emphasizing the samples closest to satisfaction, while the integral branch emphasizes the larger positive margins. 

\vspace{-0.2cm}
In summary, the augmentation-based continuous-time realization preserves the main optimization advantage of GMSR. Gradient information is distributed over multiple relevant subformulas and time instants, with stronger emphasis placed on the most critical active margins rather than a single extremizing sample.

\vspace{-0.2cm}
\subsection{Time-dilation}\label{ssec:dilation}
\vspace{-0.2cm}
We employ time-dilation \cite{kamath2023seco} to transform the free-final-time CT-STL-augmented optimal control problem \eqref{eq:ct_ocp_aug} into an equivalent fixed-final-time formulation. Specifically, we define a strictly increasing, continuously differentiable mapping
\vspace{-0.2cm}
\begin{align*}
    \tilde t:[0,1]\to\mathbb R_+,
\end{align*}
which maps the fixed dilated-time domain $\tau\in[0,1]$ to the physical time $t$, with
$
    \tilde t(0)=0,
    \,
    \tilde t(1)=t_{\mathrm f}.
$
The rate of time flow is
\vspace{-0.2cm}
\begin{align*}
    s(\tau):=\frac{\mathrm d\tilde t(\tau)}{\mathrm d\tau},
\end{align*}
and is treated as an additional decision variable.

\vspace{-0.2cm}
To express the CT-STL-augmented dynamics in the dilated-time domain, we append the physical time as an additional state and the dilation factor as an additional input. For a given temporal formula $\psi$, define
\vspace{-0.2cm}
\begin{align*}
    \tilde x(\tau)
    &:=
    \begin{bmatrix}
        x^\psi(\tilde t(\tau))\\
        \tilde t(\tau)
    \end{bmatrix},
    \qquad
    \tilde u(\tau)
    :=
    \begin{bmatrix}
        u(\tilde t(\tau))\\
        s(\tau)
    \end{bmatrix}.
\end{align*}
Then, if the physical-time augmented dynamics are
\vspace{-0.2cm}
\begin{align*}
    \dot x^\psi(t)=F^\psi\big(t,x^\psi(t),u(t)\big),
\end{align*}
the corresponding dynamics in the $\tau$-domain are
\vspace{-0.2cm}
\begin{align*}
    \frac{\mathrm d\tilde x(\tau)}{\mathrm d\tau}
     =& f\big(\tilde x(\tau),\tilde u(\tau)\big)\\
    :=&
    \begin{bmatrix}
        s(\tau)\,F^\psi\big(\tilde t(\tau),x^\psi(\tilde t(\tau)),u(\tilde t(\tau))\big)\\
        s(\tau)
    \end{bmatrix}.
\end{align*}
Thus, the final dilated time is fixed at
$
    \tau_{\mathrm f}=1,
$
while the physical final time remains free through
\vspace{-0.2cm}
\begin{align*}
    t_{\mathrm f}
    =
    \tilde t(1)
    =
    \int_0^1 s(\tau)\,\mathrm d\tau.
\end{align*}
The augmented boundary conditions are mapped equivalently into the dilated-time domain as
\vspace{-0.2cm}
\begin{align*}
    \tilde P^\psi\big(\tilde x(0),\tilde u(0),\tilde x(1),\tilde u(1)\big)
    &:=
    P^\psi\big(z_0^\psi,z^\psi_{\mathrm f}\big), \\
    \tilde Q^\psi\big(\tilde x(0),\tilde u(0),\tilde x(1),\tilde u(1)\big)
    &:=
    Q^\psi\big(z_0^\psi,z^\psi_{\mathrm f}\big),
\end{align*}
where $z_0^\psi := \big( x^\psi(0),u(0) \big)$ and $z^\psi_{\mathrm f} := \big( x^\psi(t_{\mathrm f}),u(t_{\mathrm f}) \big)$. Here, $x^\psi(0)$, $u(0)$, $x^\psi(t_{\mathrm f})$, and $u(t_{\mathrm f})$ are understood as the corresponding components of $\tilde x(0)$, $\tilde u(0)$, $\tilde x(1)$, and $\tilde u(1)$, respectively. Hence, the free-final-time CT-STL-augmented problem is equivalently written as a fixed-final-time problem on the interval $[0,1]$.

\vspace{-0.2cm}
\subsection{Control Input Parameterization}\label{ssec:control_param}
\vspace{-0.2cm}
To obtain a finite-dimensional optimization problem, we parameterize the continuous-time input trajectory using a finite set of nodal values. We discretize the dilated-time interval $ [0,1] $ into $K$ uniformly spaced nodes
\vspace{-0.2cm}
\begin{align*}
    0=\tau_1<\cdots<\tau_K=1,
\end{align*}
with
$
    \Delta\tau:=\tau_{k+1}-\tau_k,
    \,
    k\in\mathcal K:=\{1,\dots,K-1\}.
$

\vspace{-0.2cm}
The control input is parameterized by first-order hold (FOH):
\vspace{-0.2cm}
\begin{align*}
    u(\tau)
    =
    u_k \frac{\tau_{k+1}-\tau}{\Delta\tau}
    +
    u_{k+1} \frac{\tau-\tau_k}{\Delta\tau},
    \qquad \tau\in[\tau_k,\tau_{k+1}],
\end{align*}
for each $k\in\mathcal K$.

\vspace{-0.2cm}
In contrast, the dilation factor is parameterized by zero-order hold (ZOH):
\vspace{-0.2cm}
\begin{align*}
    s(\tau)=s_k, \qquad \tau\in[\tau_k,\tau_{k+1}),
\end{align*}
so that
\vspace{-0.2cm}
\begin{align*}
    \Delta t_k = \int_{\tau_k}^{\tau_{k+1}} s(\tau)\,\mathrm d\tau = s_k \Delta\tau.
\end{align*}
To ensure strictly forward time evolution, we impose
\vspace{-0.2cm}
\begin{align*}
    s_k \ge s_{\min}, \qquad \forall k\in\mathcal K,
\end{align*}
for some prescribed $s_{\min}>0$.

\vspace{-0.2cm}
A brief discussion of why the dilation factor $s(\tau)$ is parameterized by ZOH rather than FOH is deferred to Appendix~\ref{app:foh_dilation}.

\vspace{-0.2cm}
\subsection{Discretization and Multiple Shooting}\label{ssec:discretization}
\vspace{-0.2cm}
With the time-dilated augmented state and input values $\tilde{x}_k$ and $\tilde{u}_k$ at the nodes $\tau_k$ acting as decision variables, we discretize the dynamics using a multiple-shooting method~\cite{bock1984multiple}. The state transition from node $k$ to node $k+1$ is enforced through
\vspace{-0.2cm}
\begin{align*}
    \tilde{x}_{k+1}=f_{\tau_k}^{\tau_{k+1}}(\tilde{x}_k,\tilde{u}_k,\tilde{u}_{k+1}),
\end{align*}
where the flow map is defined by
\vspace{-0.2cm}
\begin{align*}
    f_{\tau_k}^{\tau_{k+1}}(\tilde{x}_k,\tilde{u}_k,\tilde{u}_{k+1})
    :=
    \tilde{x}_k
    +
    \int_{\tau_k}^{\tau_{k+1}}
    f\big(\tilde{x}(\tau),\tilde{u}(\tau)\big)\,\mathrm d\tau.
\end{align*}
Here, $\tilde{x}(\tau)$ denotes the time-dilated augmented state trajectory evolving from the initial condition $\tilde{x}(\tau_k)=\tilde{x}_k$ under the parameterized input $\tilde{u}(\tau)$. The flow map is evaluated numerically using a chosen integration scheme, yielding a finite-dimensional nonlinear program with multiple-shooting defect constraints. 

\vspace{-0.2cm}
In the prox-convex iterations, the flow map and defect constraints are linearized with respect to the local shooting variables. These Jacobians are obtained by integrating the variational equations of the time-dilated dynamics on each interval; see Appendix~\ref{app:flow_map_linearization}.

\vspace{-0.2cm}
\subsection{Finite-dimensional nonlinear optimal control problem}\label{ssec:fd_nlp}
\vspace{-0.2cm}
After time-dilation, input parameterization, dynamics discretization, and CT-STL realization, the original continuous-time problem is transcribed into a finite-dimensional nonlinear optimal control problem. Let
\vspace{-0.2cm}
\begin{align*}
    Z
    :=
    \big(
    \tilde{x}_1,\dots,\tilde{x}_K,\,
    \tilde{u}_1,\dots,\tilde{u}_K
    \big)
\end{align*}
collect the nodal decision variables, where $\tilde{x}_k$ denotes the time-dilated state at node $\tau_k$ and $\tilde{u}_k$ denotes the corresponding nodal input.

\vspace{-0.2cm}
The original cost induces a finite-dimensional cost function, which we denote by
\vspace{-0.2cm}
\begin{align*}
    L_d(Z):=L(t_{\mathrm f},x_K).
\end{align*}
Define the multiple-shooting defect on interval $k\in\{1,\dots,K-1\}$ by
\vspace{-0.2cm}
\begin{align*}
    D_k(Z)
    &:=
    \tilde{x}_{k+1}
    -
    f_{\tau_k}^{\tau_{k+1}}(\tilde{x}_k,\tilde{u}_k,\tilde{u}_{k+1}),
\end{align*}
and let
$
    D(Z)
    :=
    \big(
    D_1(Z),\dots,D_{K-1}(Z)
    \big)
$
denote the stacked dynamics-defect map. Let
\begin{align*}
    P_d(Z)\le 0,
    \qquad
    Q_d(Z)=0
\end{align*}
denote the finite-dimensional inequality and equality maps induced by the augmented boundary maps $P^\psi$ and $Q^\psi$ after time-dilation, input parameterization, and discretization. 

For the CT-STL specification, we write
\vspace{-0.2cm}
\begin{align*}
    \Gamma_{\bm c}^{\phi}[Z]
\end{align*}
for the scalar robustness functional associated with the full specification $\phi$, evaluated on the discretized trajectory induced by $Z$. In the augmentation-based realization used in the main development, $\Gamma_{\bm c}^{\phi}[Z]$ is obtained from the terminal values of the auxiliary states induced by $Z$. The exact dense-time realization summarized in Appendix~\ref{ssec:exact_dense_ctstl} provides an alternative instantiation of the same robustness functional by evaluating GMSR temporal operators on the dense integration-grid states.

\vspace{-0.2cm}
With these definitions, the finite-dimensional nonlinear optimal control problem is written as
\vspace{-0.2cm}
\begin{subequations}\label{eq:fd_nlp}
\begin{align}
    \underset{Z\in\mathcal Z}{\mathrm{minimize}}
    \quad
    &L_d(Z)-w_{\phi}\Gamma_{\bm c}^{\phi}[Z]
    \label{eq:fd_nlp_cost} \\
    \mathrm{subject\ to}\quad
    &D(Z)=0, \label{eq:fd_nlp_dyn} \\
    &P_d(Z)\le 0, \label{eq:fd_nlp_ineq} \\
    &Q_d(Z)=0, \label{eq:fd_nlp_eq}
\end{align}
\end{subequations}
where $w_{\phi}\ge 0$ is the weight assigned to the CT-STL robustness term, and $\mathcal Z$ denotes any explicitly retained convex constraints on the discretized variables, such as simple bounds on states, inputs, and time-dilation variables.

\vspace{-0.2cm}
Problem \eqref{eq:fd_nlp} is generally nonconvex because of the discretized dynamics, the finite-dimensional boundary mappings $P_d,Q_d$, the CT-STL robustness functional $\Gamma_{\bm c}^{\phi}[Z]$, and possibly the discretized terminal cost $L_d(Z)$ itself.

\vspace{-0.2cm}
\subsection{Prox-convex algorithm for solution}\label{ssec:proxconvex_solution}
\vspace{-0.2cm}
To solve the finite-dimensional nonlinear optimal control problem \eqref{eq:fd_nlp}, we employ a convergence-guaranteed SCP method, prox-convex algorithm~\cite{proxconvex}. The key idea is to separate the discretized problem into a convex part that is kept exact, a nonconvex objective part, represented in smooth-composite form, and a collection of nonconvex constraints handled through an exact penalty.

\vspace{-0.2cm}
We write the finite-dimensional objective as
\vspace{-0.2cm}
\begin{align*}
    L_d(Z)-w_\phi \Gamma_{\bm c}^{\phi}[Z]
    =
    G(Z)+S(R(Z)),
\end{align*}
where $G$ denotes the convex part of the objective, $S:\mathbb{R}^m\to\mathbb{R}$ is a smooth outer map, and
$
    R(Z)=\big(R_1(Z),\dots,R_m(Z)\big)
$
collects convex inner channels. This representation is natural, for example, when convex predicate functions are composed through smooth GMSR operators. In the simplest case, when no additional structure is exploited, one may simply absorb the entire nonconvex objective into the term $S(R(Z))$.

\vspace{-0.2cm}
Next, we collect the remaining nonconvex constraints into
\vspace{-0.2cm}
\begin{align*}
    C(Z)
    :=
    \begin{bmatrix}
        D(Z)\\
        P_d(Z)\\
        Q_d(Z)
    \end{bmatrix},
\end{align*}
where $D(Z)=0$ denotes the multiple-shooting defect constraints, $P_d(Z)\le 0$ the discretized inequality constraints, and $Q_d(Z)=0$ the discretized equality constraints. These constraints are penalized through the convex exact-penalty functional
\vspace{-0.2cm}
\begin{align*}
    H(C(Z))
    :={}&
    w_{\mathrm{dyn}}\|D(Z)\|_1
    \\
    &+
    w_P \bm{1}^{\top} [P_d(Z)]_+
    \\
    &+
    w_Q\|Q_d(Z)\|_1 .
\end{align*}
where $w_{\mathrm{dyn}},w_P,w_Q>0$ are penalty weights. The resulting penalized nonlinear objective is therefore
\vspace{-0.2cm}
\begin{align*}
    J_{\mathrm{nl}}(Z)
    :=
    G(Z)+S(R(Z))+H(C(Z)).
\end{align*}
At iteration $j+1$, given the current iterate $Z^j$, the nonconvex objective term $S(R(Z))$ is replaced by its prox-convex local model. Define
\vspace{-0.2cm}
\begin{align*}
    \mathcal I_j^-
    :=
    \big\{
        i\in\{1,\dots,m\}
        :
        \nabla_i S(R(Z^j))<0
    \big\}.
\end{align*}
Then the local convex model of $S(R(Z))$ at $Z^j$ is
\vspace{-0.2cm}
\begin{align*}
    \mathcal M_S^j (Z)
    &:=
    S(R(Z^j))
    +
    \sum_{i=1}^{m}
    \nabla_i S(R(Z^j))\,\Phi_i(Z;Z^j),
\end{align*}
where
\vspace{-0.2cm}
\begin{align*}
    \Phi_i(Z;Z^j)
    :=
    \begin{cases}
        R_i(Z)-R_i(Z^j),
        & \text{if } i\notin\mathcal I_j^-, \\[0.1cm]
        \nabla R_i(Z^j)^\top (Z-Z^j),
        & \text{if } i\in\mathcal I_j^-.
    \end{cases}
\end{align*}
Hence, channels whose outer derivative is nonnegative are kept exact, while channels with negative outer derivative are affine linearized in order to preserve convexity of the local model. This is precisely the prox-convex mechanism that allows a smooth outer map to be combined with convex inner channels without destroying convexity of the subproblem~\cite{proxconvex}.

\vspace{-0.2cm}
For the penalty term, we use the standard convex-composite linearization
\vspace{-0.2cm}
\begin{align*}
    \mathcal M_H^j (Z)
    &:=
    H\!\left(
        C(Z^j)+\nabla C(Z^j)(Z-Z^j)
    \right).
\end{align*}
The convex subproblem solved at iteration $j+1$ is then
\vspace{-0.2cm}
\begin{align*}
    \underset{Z\in\mathcal Z}{\mathrm{minimize}}
    \;\,
    G(Z)
    \!+\!
    \mathcal M_S^j (Z)
    \!+\!
    \mathcal M_H^j (Z)
    \!+\!
    \frac{1}{2}\|Z-Z^j\|_{Q_j}^2,
\end{align*}
where $Q_j\succ0$ is the proximal metric. Since $G$ and $\mathcal Z$ are convex, $\mathcal M_S(\cdot;Z^j)$ is convex by construction, and $H$ is convex, the subproblem is convex.

\vspace{-0.2cm}
The proximal metric $Q_j$ may be chosen simply as a scaled identity, or, if desired, enriched with positive-semidefinite curvature information. In particular, one may take
\vspace{-0.2cm}
\begin{align*}
    Q_j=\mu_j I + B_j^+,
\end{align*}
where $\mu_j>0$ is a scalar proximal weight and $B_j^+\succeq0$ is a PSD curvature approximation. Additional second-order information can therefore be embedded into the convex subproblem if desired, but we do not pursue those details here and instead refer the reader to~\cite{proxconvex}.

\vspace{-0.2cm}
The scalar proximal weight is updated adaptively using the ratio between the actual and predicted decreases in the penalized objective, following the adaptive variable-metric prox-convex framework of~\cite{proxconvex}. Thus, when the local model is accurate, the method reduces the proximal regularization and takes bolder steps; when the agreement deteriorates, the proximal weight is increased to restore reliability. Under the standing assumptions of the prox-convex method, this yields a convergence-guaranteed SCP scheme for the penalized problem.

\vspace{-0.2cm}
\section{Numerical Results}\label{sec:numerical}
\vspace{-0.2cm}
We demonstrate the proposed framework on four trajectory-generation problems. The first three are fixed-final-time double-integrator examples constructed to illustrate the \always{}, \eventually{}, and \until{} operators in isolation. The fourth is a free-final-time $6$-DoF quadrotor problem with moving waypoints, moving implication-type obstacle constraints, and terminal-set requirements. In all cases, the primary objective is specification satisfaction. To keep the presentation focused, the exact code-level scaling and smoothing constants used in the implementation are omitted here.

\vspace{-0.2cm}
In our implementation, \texttt{JAX} is used to integrate the augmented dynamics and to compute the first-order derivative information of the smooth residual maps required for linearization~\cite{bradbury2018jax}. Each convex subproblem is modeled in \texttt{CVXPY}~\cite{diamond2016cvxpy} and solved using \texttt{QOCO}~\cite{chari2025qoco}. The implementation is available at \url{https://github.com/UW-ACL/TrajOpt_CT-STL}

\vspace{-0.2cm}
\begin{rem}[Satisfaction-oriented surrogates]
In some of the examples below, the objective is only to generate trajectories that satisfy the target specification. Accordingly, the auxiliary-state constructions are chosen to capture the relevant {satisfaction mechanism} of each operator, rather than to reproduce the full CT-GMSR robustness exactly. We therefore denote the terminal objective quantities by $\widehat{\Gamma}$ rather than by $\Gamma$. These surrogates are sign- and satisfaction-oriented, but need not coincide with the full CT-GMSR robustness value.
\end{rem}

\vspace{-0.2cm}
\subsection{Double-integrator examples}\label{sec:num_di}
\vspace{-0.2cm}
\subsubsection{Base dynamics and common continuous-time path constraints}
\vspace{-0.2cm}
For the three double-integrator examples, let
$
    x := (r,v)\in\mathbb R^6,
    \,
    r=(r_x,r_y,r_z),
    \,
    v=(v_x,v_y,v_z),
    \,
    u=(u_x,u_y,u_z)\in\mathbb R^3,
$
with dynamics
$
    \dot r(t)=v(t),
    \,
    \dot v(t)=u(t)-g_0e_3,
    \,
    g_0=9.806~\mathrm{m/s^2}.
$

\vspace{-0.2cm}
All three examples are subject to the same continuous-time tilt, thrust, and speed constraints,
$
    u_x^2+u_y^2 \le (\cos\theta_{\max}\,u_z)^2,
    \,
    \|u\|_2 \le T_{\max},
    \,
    \|v\|_2 \le v_{\max},
$
with $\theta_{\max}=45^\circ$, $T_{\max}=1.75\,g_0$, and $v_{\max}=6~\mathrm{m/s}$. Defining the predicates
\vspace{-0.2cm}
\begin{align*}
    \mu_{\theta}(u)
    &:=
    \Big(
        (\cos\theta_{\max}\,u_z)^2-u_x^2-u_y^2 \ge 0
    \Big),\\
    \mu_T(u)
    &:=
    \Big(
        T_{\max}^2-u_x^2-u_y^2-u_z^2 \ge 0
    \Big),\\
    \mu_v(v)
    &:=
    \Big(
        v_{\max}^2-v_x^2-v_y^2-v_z^2 \ge 0
    \Big),
\end{align*}
the common path-constraint formula is
\vspace{-0.2cm}
\begin{align*}
    \varphi_{\mathrm{path}}
    :=
    \bm G_{[0,t_{\mathrm f}]}
    \big(
        \mu_{\theta}\wedge\mu_T\wedge\mu_v
    \big).
\end{align*}
In the implementation, these continuous-time \always{} constraints are enforced through an accumulated path-penalty state $\eta_{\mathrm p}$,
\vspace{-0.2cm}
\begin{align*}
    \dot \eta_{\mathrm p}(t)
    =
    \chi_{\theta}(u(t))
    +
    \chi_T(u(t))
    +
    \chi_v(v(t)),
    \qquad
    \eta_{\mathrm p}(0)=0,
\end{align*}
where
\vspace{-0.2cm}
\begin{align*}
    \chi_{\theta}(u)
    &:=
    \big[
        (\cos\theta_{\max}\,u_z)^2-u_x^2-u_y^2
    \big]_-^2,\\
    \chi_T(u)
    &:=
    \big[
        T_{\max}^2-u_x^2-u_y^2-u_z^2
    \big]_-^2,\\
    \chi_v(v)
    &:=
    \big[
        v_{\max}^2-v_x^2-v_y^2-v_z^2
    \big]_-^2,
\end{align*}
and $[s]_+ := \max(s,0)$, $[s]_- := \min(s,0)$. Hence, $\eta_{\mathrm p}(t_{\mathrm f})=0$ certifies satisfaction of the common continuous-time path constraints.

\vspace{-0.2cm}
To avoid repeating the boundary conditions, define
\vspace{-0.2cm}
\begin{align*}
    \mathcal X_{\mathrm i}(r_{\mathrm i})
    &:=
    \left\{
        (r,v)\in\mathbb R^6 : r=r_{\mathrm i},\; v=0
    \right\},\\
    \mathcal X_{\mathrm f}(r_{\mathrm f})
    &:=
    \left\{
        (r,v)\in\mathbb R^6 : r=r_{\mathrm f},\; v=0
    \right\},
\end{align*}
and let
\vspace{-0.2cm}
\begin{align*}
    u_{\mathrm h}:=(0,0,g_0)
\end{align*}
denote the gravity-compensating input. Then the common endpoint conditions are
\vspace{-0.2cm}
\begin{align*}
    x(0)\in\mathcal X_{\mathrm i}(r_{\mathrm i}),
    \quad
    x(t_{\mathrm f})\in\mathcal X_{\mathrm f}(r_{\mathrm f}),
    \quad
    u(0)=u(t_{\mathrm f})=u_{\mathrm h}.
\end{align*}

\vspace{-0.2cm}
\subsubsection{Example 1: \always{} obstacle avoidance}
\vspace{-0.2cm}
The first example uses
$
    r_{\mathrm i}=(-5,0,0),
    \,
    r_{\mathrm f}=(5,0,0),
    \,
    t_{\mathrm f}=7~\mathrm{s},
    \,
    K=6,
$
and requires the vehicle to always avoid two implication-type obstacle regions. The obstacle specifications are
\vspace{-0.2cm}
\begin{align*}
    \varphi_{\mathrm{obs},1}(r)
    &:=
    \left(
        a_1 \le r_x \le b_1
        \;\Rightarrow\;
        r_y \le c_1
    \right),\\
    \varphi_{\mathrm{obs},2}(r)
    &:=
    \left(
        a_2 \le r_x \le b_2
        \;\Rightarrow\;
        r_y \ge c_2
    \right),
\end{align*}
with
$
    (a_1,b_1,c_1)=(-3.25,-0.25,-2),
    \,
    (a_2,b_2,c_2)=(0.25,3.25,2).
$
Thus, the overall task is
\vspace{-0.2cm}
\begin{align*}
    \varphi_{\mathrm{alw}}
    :=
    \varphi_{\mathrm{path}}
    \wedge
    \bm G_{[0,t_{\mathrm f}]}
    \big(
        \varphi_{\mathrm{obs},1}\wedge\varphi_{\mathrm{obs},2}
    \big).
\end{align*}
\vspace{-0.2cm}
We augment the state as
\vspace{-0.1cm}
\begin{align*}
    x^{\varphi_{\mathrm{alw}}}:=(x,\eta_{\mathrm p},\xi_1,\xi_2),
\end{align*}
where $\eta_{\mathrm p}$ handles the common path constraints and $\xi_1,\xi_2$ accumulate obstacle violations:
\vspace{-0.2cm}
\begin{align*}
    \dot \xi_1(t) &= [a_1-r_x]_-^2\,[r_x-b_1]_-^2\,[c_1-r_y]_-^2,
    \quad
    \xi_1(0)=0,\\
    \dot \xi_2(t) &= [a_2-r_x]_-^2\,[r_x-b_2]_-^2\,[r_y-c_2]_-^2,
    \quad
    \xi_2(0)=0.
\end{align*}
We therefore define the terminal surrogate
\vspace{-0.2cm}
\begin{align*}
    \widehat{\Gamma}_{\mathrm{alw}}
    &:=
    -
    \Big(
        w_{\eta}\eta_{\mathrm p}(t_{\mathrm f})
        +w_1\xi_1(t_{\mathrm f})
        +w_2\xi_2(t_{\mathrm f})
    \Big),
\end{align*}
with positive weights $w_{\eta},w_1,w_2$. Maximizing $\widehat{\Gamma}_{\mathrm{alw}}$, or equivalently minimizing its negative, favors trajectories that satisfy both the common path constraints and the obstacle-avoidance requirement.

\vspace{-0.2cm}
The resulting continuous-time problem is
\vspace{-0.2cm}
\begin{empheq}[box=\fbox]{equation*}
\begin{aligned}
    \underset{x^{\varphi_{\mathrm{alw}}}(\cdot),u(\cdot)}{\mathrm{maximize}}
    \quad &
    \widehat{\Gamma}_{\mathrm{alw}} \\
    \mathrm{subject\ to}\quad
    &
    \dot r(t)=v(t),\qquad
    \dot v(t)=u(t)-g_0e_3,\\
    &\dot \eta_{\mathrm p}(t)=\chi_{\theta}(u(t))+\chi_T(u(t))+\chi_v(v(t)),\\
    &\dot \xi_1(t) = [a_1-r_x]_-^2\,[r_x-b_1]_-^2\,[c_1-r_y]_-^2,\\
    &\dot \xi_2(t) = [a_2-r_x]_-^2\,[r_x-b_2]_-^2\,[r_y-c_2]_-^2,\\
    &
    x(0)\in\mathcal X_{\mathrm i}(r_{\mathrm i}),\qquad
    x(t_{\mathrm f})\in\mathcal X_{\mathrm f}(r_{\mathrm f}),\\
    &
    u(0)=u(t_{\mathrm f})=u_{\mathrm h}, \qquad \eta_{\mathrm p}(0) = 0,\\
    & \xi_1(0)=0, \qquad \xi_2(0)=0.
\end{aligned}
\end{empheq}

\vspace{-0.2cm}
Figure~\ref{fig:di_always} shows the optimized trajectory for the \always{} obstacle-avoidance example. The vehicle departs from the straight start-to-goal path, bends around the two implication-type forbidden regions, and passes through the admissible corridor before reaching the terminal state. 
The color-coded speed profile along the trajectory also indicates a smooth acceleration and deceleration pattern, with no abrupt maneuvers required to satisfy the obstacle-avoidance task.
Figure~\ref{fig:always_distances} confirms this behavior quantitatively. The obstacle-clearance signals remain in the safe region for the entire horizon, so the trajectory never enters either forbidden set. The common continuous-time tilt, thrust, and speed constraints are also satisfied throughout the maneuver.

\vspace{-0.2cm}
\begin{figure}[t]
\centerline{\includegraphics[scale=0.48]{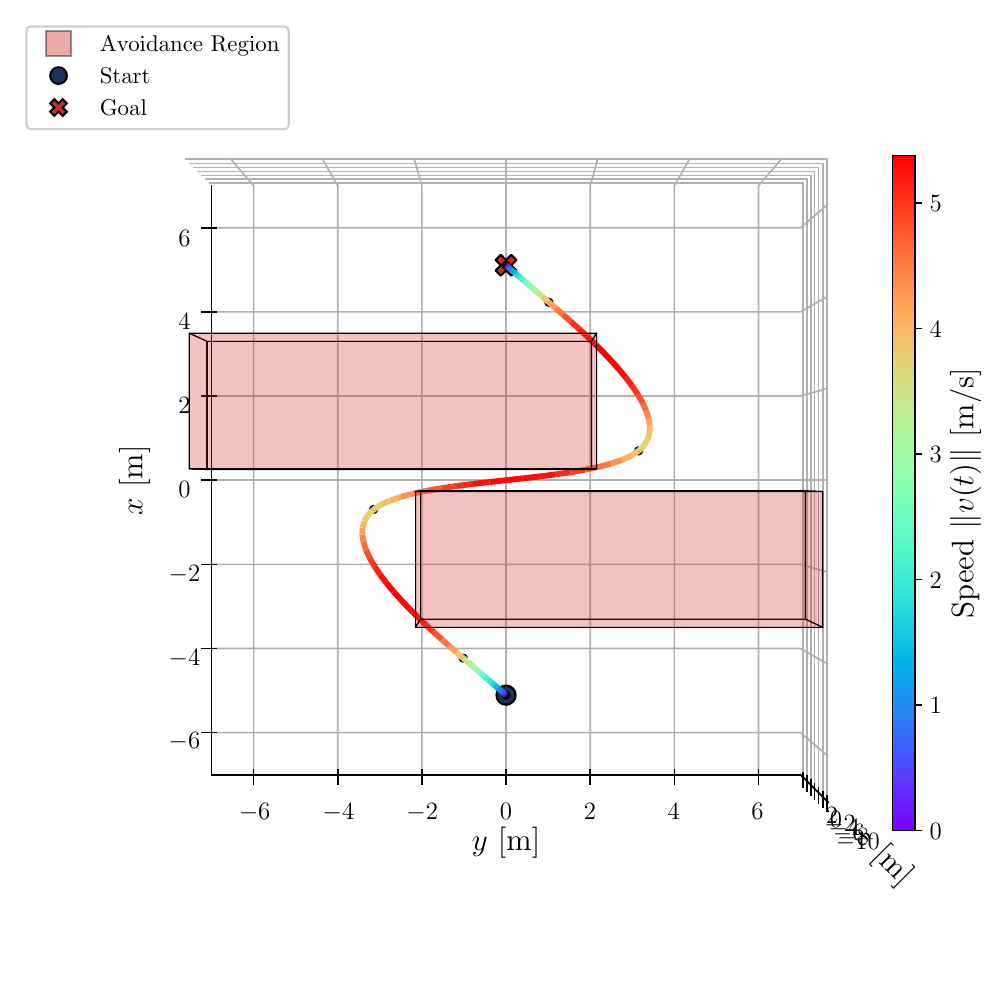}}
\caption{Optimized trajectory for the \always{} obstacle-avoidance example. The vehicle departs from the direct start-to-goal path, passes through the admissible corridor between the two forbidden rectangular regions, and reaches the terminal state without entering either avoidance region. The trajectory is color-coded by speed.}
\label{fig:di_always}
\end{figure}

\begin{figure}[t]
\centerline{\includegraphics[scale=0.45]{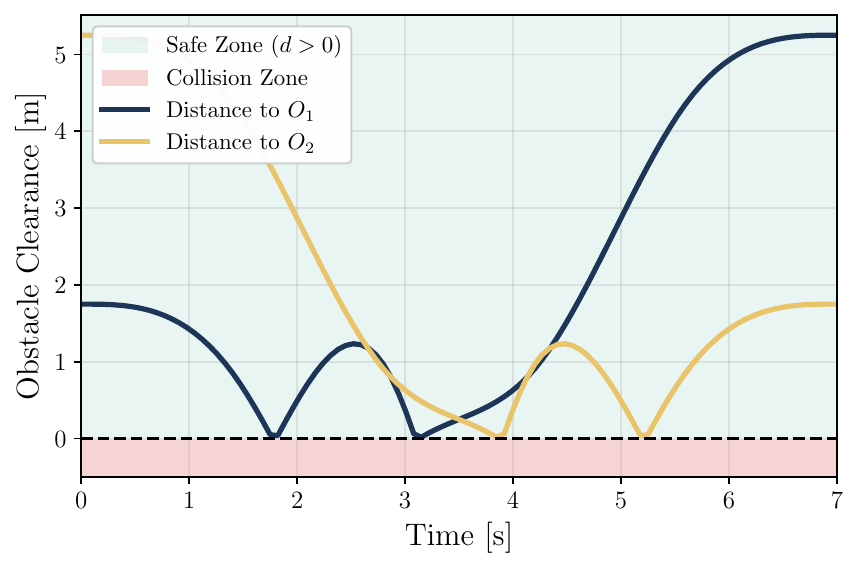}}
\caption{Obstacle-clearance histories for the \always{} example. Both clearance signals remain in the safe region throughout the maneuver, confirming continuous-time satisfaction of the two obstacle-avoidance constraints.}
\label{fig:always_distances}
\end{figure}

\vspace{-0.2cm}
\subsubsection{Example 2: \eventually{} visit three waypoints}
\vspace{-0.2cm}
The second example uses
$
    r_{\mathrm i}=(-10,0,0),
    \,
    r_{\mathrm f}=(6,0,0),
    \,
    t_{\mathrm f}=12~\mathrm{s},
    \,
    K=7,
$
and requires the vehicle to visit three waypoint regions centered at
$
    p_{w,1}=(-8,-5,0),\,
    p_{w,2}=(-6,5,0),\,
    p_{w,3}=(-4,-5,0),
$
all with radius $r_w=0.5$. Define the waypoint margins
\vspace{-0.2cm}
\begin{align*}
    \rho_i(r)
    :=
    r_w^2-\|r-p_{w,i}\|_2^2,
    \qquad i=1,2,3,
\end{align*}
and the corresponding predicates
\vspace{-0.2cm}
\begin{align*}
    \mu_{w,i}(r):=(\rho_i(r)\ge 0),
    \qquad i=1,2,3.
\end{align*}
Then the temporal task is
\vspace{-0.2cm}
\begin{align*}
    \varphi_{\mathrm{evt}}
    :=
    \varphi_{\mathrm{path}}
    \wedge
    \bm F_{[0,t_{\mathrm f}]}\mu_{w,1}
    \wedge
    \bm F_{[0,t_{\mathrm f}]}\mu_{w,2}
    \wedge
    \bm F_{[0,t_{\mathrm f}]}\mu_{w,3}.
\end{align*}
We augment the state as
\vspace{-0.2cm}
\begin{align*}
    x^{\varphi_{\mathrm{evt}}}
    :=
    (x,\eta_{\mathrm p},y_1,z_1,y_2,z_2,y_3,z_3),
\end{align*}
where $\eta_{\mathrm p}$ again enforces the common path constraints, and for each waypoint we introduce a geometric-integral state $y_i$ and an additive state $z_i$:
\vspace{-0.2cm}
\begin{align*}
    \dot y_i(t) &= \frac{y_i(t)}{t_{\mathrm f}}\,\log\!\Big(
        \varepsilon_{\mathrm{evt}}
        +[\rho_i(r)]_-^2
    \Big),
    \qquad
    y_i(0)=1,\\
    \dot z_i(t) &= \frac{1}{t_{\mathrm f}}[\rho_i(r)]_+^2,
    \qquad
    z_i(0)=0,
\end{align*}
The state $y_i$ acts as a geometric-integral measure of how strongly the trajectory stays outside waypoint $i$, while $z_i$ accumulates only when the vehicle is inside waypoint $i$. We define the corresponding waypoint-wise terminal surrogates
\vspace{-0.2cm}
\begin{align*}
    \widehat{\Gamma}_{w,i}
    :=
    \big(c_{\mathrm{evt}}^2+\alpha_i z_i(t_{\mathrm f})\big)^{1/2}
    -
    \big(c_{\mathrm{evt}}^2+\beta_i y_i(t_{\mathrm f})\big)^{1/2},
\end{align*}
for $i=1,2,3$ with positive constants $c_{\mathrm{evt}},\alpha_i,\beta_i$. Thus, entering waypoint $i$ increases $\widehat{\Gamma}_{w,i}$ through $z_i$, whereas remaining outside decreases it through $y_i$.

\vspace{-0.2cm}
The resulting problem is
\begin{empheq}[box=\fbox]{equation*}
\begin{aligned}
    \underset{x^{\varphi_{\mathrm{evt}}}(\cdot),u(\cdot)}{\mathrm{maximize}}
    \quad &
    -w_{\eta}\eta_{\mathrm p}(t_{\mathrm f})
    +
    \sum_{i=1}^{3}\widehat{\Gamma}_{w,i} \\
    \mathrm{subject\ to}\quad
    &
    \dot r(t)=v(t),\qquad
    \dot v(t)=u(t)-g_0e_3,\\
    &
    \dot \eta_{\mathrm p}(t)=\chi_{\theta}(u(t))+\chi_T(u(t))+\chi_v(v(t)),\\
    & \dot y_i(t) = \frac{y_i(t)}{t_{\mathrm f}}\,\log\!\Big( \varepsilon_{\mathrm{evt}} + [\rho_i(r)]_-^2 \Big), \quad \forall i,\\
    & \dot z_i(t) = \frac{1}{t_{\mathrm f}}[\rho_i(r)]_+^2, \quad \forall i,\\
    &
    x(0)\in\mathcal X_{\mathrm i}(r_{\mathrm i}),\qquad
    x(t_{\mathrm f})\in\mathcal X_{\mathrm f}(r_{\mathrm f}),\\
    & u(0)=u(t_{\mathrm f})=u_{\mathrm h}, \qquad  \eta_{\mathrm p}(0)=0, \\
    & y_i(0)=1, \qquad z_i(0)=0, \qquad \forall i.
\end{aligned}
\end{empheq}

\vspace{-0.2cm}
Figure~\ref{fig:di_eventually} shows the optimized trajectory for the three-waypoint \eventually{} example. Rather than proceeding directly to the goal, the trajectory bends toward each waypoint region in turn and then continues to the terminal state. 
The result highlights the role of the continuous-time event formulation. Satisfaction is achieved through the actual continuous-time trajectory rather than being tied to a small set of coarse discretization nodes.
Figure~\ref{fig:eventually_distances} verifies that all three distance traces enter the satisfaction zone $d\le r_w$ at some time in the interval. This behavior is exactly consistent with the three \eventually{} requirements.

\begin{figure}[t]
\centerline{\includegraphics[scale=0.50]{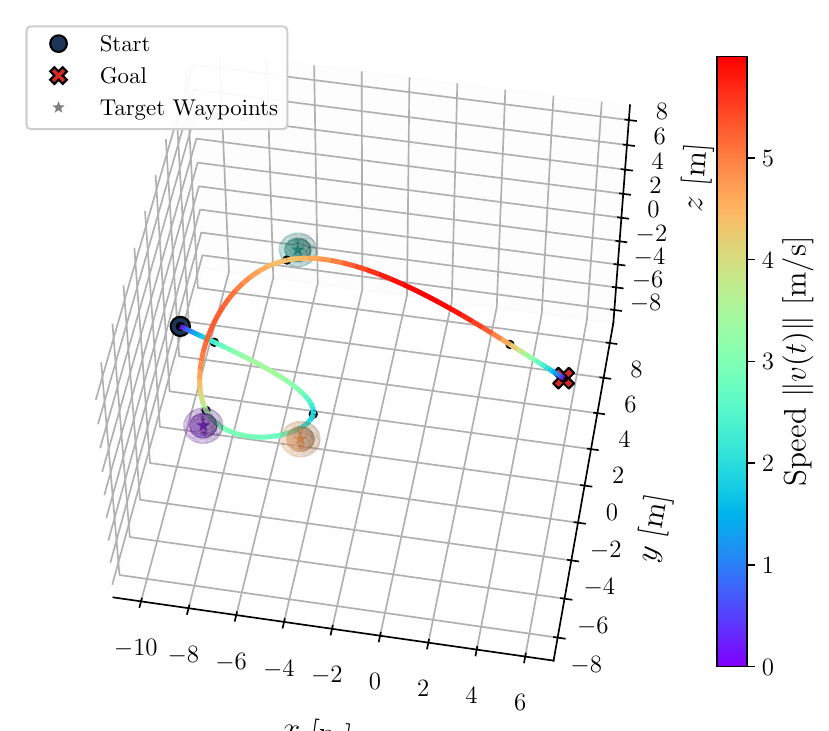}}
\caption{Optimized trajectory for the three-waypoint \eventually{} example. The vehicle successively bends toward and visits all three waypoint regions before proceeding to the terminal state. The trajectory is color-coded by speed.}
\label{fig:di_eventually}
\end{figure}

\begin{figure}[t]
\centerline{\includegraphics[scale=0.43]{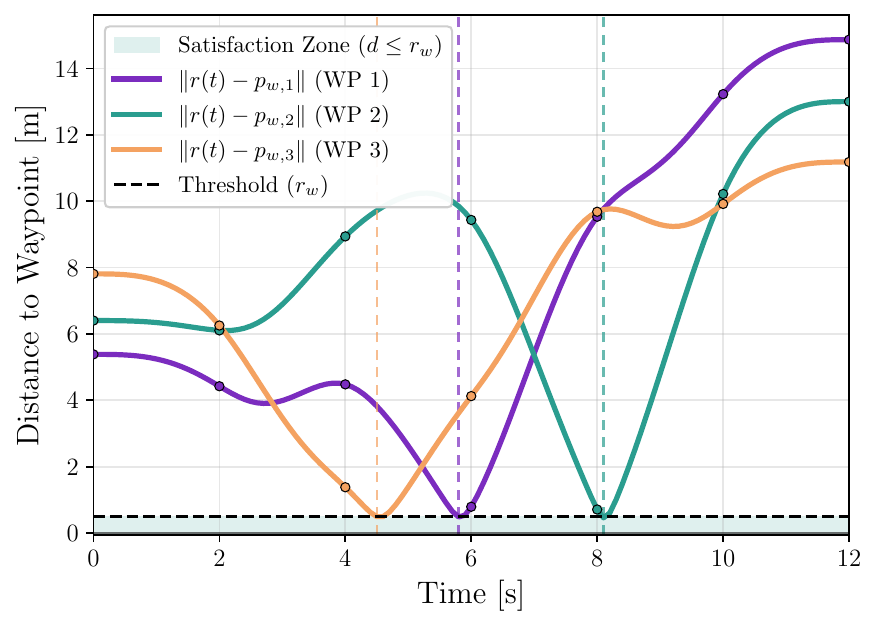}}
\caption{Distances from the vehicle to the three waypoint centers in the \eventually{} example. Each trace enters the satisfaction zone $d\le r_w$ at some time in the interval, confirming continuous-time satisfaction of all three \eventually{} requirements.}
\label{fig:eventually_distances}
\end{figure}

\vspace{-0.2cm}
\subsubsection{Example 3: \until{} charging-station task}
\vspace{-0.2cm}
The third example uses
$
    r_{\mathrm i}=(-6,0,0),
    \,
    r_{\mathrm f}=(6,0,0),
    \,
    t_{\mathrm f}=5.5~\mathrm{s},
    \,
    K=6,
$
and requires the vehicle to keep its speed below a threshold until it reaches a charging-station region centered at
$
    p_{\mathrm c}=(-4,-2,0),
    \,
    d_{\mathrm c}=0.2,
    \,
    v_{\mathrm safe}=2.0~\mathrm{m/s}.
$
Define
\vspace{-0.2cm}
\begin{align*}
    \mu_{\mathrm s}(v)
    &:=
    \big(v_{\mathrm safe}^2-\|v\|_2^2 \ge 0\big),\\
    \mu_{\mathrm c}(r)
    &:=
    \big(d_{\mathrm c}^2-\|r-p_{\mathrm c}\|_2^2 \ge 0\big),
\end{align*}
so that the desired temporal task is
\vspace{-0.2cm}
\begin{align*}
    \varphi_{\mathrm{unt}}
    :=
    \varphi_{\mathrm{path}}
    \wedge
    \big(\mu_{\mathrm s}\,\bm U_{[0,t_{\mathrm f}]}\,\mu_{\mathrm c}\big).
\end{align*}
To keep this example simple, the implementation uses a reduced augmentation tailored to the satisfaction mechanism of the \until{} condition. We augment the state as
\vspace{-0.2cm}
\begin{align*}
    x^{\varphi_{\mathrm{unt}}}:=(x,\eta_{\mathrm p},y,z),
\end{align*}
where $\eta_{\mathrm p}$ handles the common path constraints, $y$ accumulates the pre-arrival speed-threshold violation, and $z$ acts as a multiplicative witness-state for the coupled \until{} condition.

\vspace{-0.2cm}
First, define
\vspace{-0.2cm}
\begin{align*}
    \dot y(t)
    &=
    \frac{1}{t_{\mathrm f}}
    \left[\,v_{\mathrm safe}^2-\|v(t)\|_2^2\,\right]_-^2,
    \quad
    y(0)=0.
\end{align*}
Thus, $y(t)$ remains small only if the speed bound is respected up to time $t$. We then introduce the normalized pre-arrival violation measure
\vspace{-0.2cm}
\begin{align*}
    q(t):=\frac{y(t)}{t+\varepsilon_t}, \text{ where $\varepsilon_t>0$}.
\end{align*}
Next, define the charging-station miss measure
\vspace{-0.2cm}
\begin{align*}
    d(r)
    :=
    \|r-p_{\mathrm c}\|_2^2-d_{\mathrm c}^2,
\end{align*}
which is negative inside the charging region and positive outside. The speed and charging conditions are coupled through
\vspace{-0.2cm}
\begin{align*}
    \chi(r,y,t)
    &:=
    \sqrt{
        C_1^2 +
        \frac{1}{2}
        \left(
            [d(r)]_+^2 + [q(t)]_+^2
        \right)
    }
    - C_1,
\end{align*}
with $C_1>0$. This quantity can be small only when the vehicle is near or inside the charging region and the pre-arrival speed-violation measure is also small.

\vspace{-0.2cm}
Finally, the multiplicative witness-state $z$ evolves according to
\vspace{-0.2cm}
\begin{align*}
    \dot z(t)
    &=
    \frac{z(t)}{t_{\mathrm f}}
    \,
    \log\!\left(
        \chi(r(t),y(t),t)^2 + \varepsilon_{\mathrm u}
    \right),
    \quad
    z(0)=1,
\end{align*}
with $\varepsilon_{\mathrm u}>0$. Since $z(0)=1$ and the dynamics are multiplicative, the terminal value $z(t_{\mathrm f})$ depends on the entire time history of the coupled miss measure $\chi$. We therefore define the terminal surrogate
\vspace{-0.2cm}
\begin{align*}
    \widehat{\Gamma}_{\mathrm{unt}}
    :=
    c_{\mathrm u}
    -
    \sqrt{c_{\mathrm u}^2+z(t_{\mathrm f})},
\end{align*}
with $c_{\mathrm u}>0$. Maximizing $\widehat{\Gamma}_{\mathrm{unt}}$, or equivalently minimizing $z(t_{\mathrm f})$, favors trajectories for which the coupled quantity $\chi$ remains small in the sense required by the \until{} semantics.

\vspace{-0.2cm}
The resulting problem is
\vspace{-0.2cm}
\begin{empheq}[box=\fbox]{equation*}
\begin{aligned}
    \underset{x^{\varphi_{\mathrm{unt}}}(\cdot),u(\cdot)}{\mathrm{maximize}}
    \quad &
    -w_{\eta}\eta_{\mathrm p}(t_{\mathrm f})
    +
    \widehat{\Gamma}_{\mathrm{unt}} \\
    \mathrm{subject\ to}\quad
    &
    \dot r(t)=v(t),\qquad
    \dot v(t)=u(t)-g_0e_3,\\
    &
    \dot \eta_{\mathrm p}(t)=\chi_{\theta}(u(t))+\chi_T(u(t))+\chi_v(v(t)),\\
    &
    \dot y(t)=\frac{1}{t_{\mathrm f}}\left[\,v_{\mathrm safe}^2-\|v(t)\|_2^2\,\right]_-^2,\\
    &
    \dot z(t)=\frac{z(t)}{t_{\mathrm f}}\, \log\!\left(\chi(r(t),y(t),t)^2+\varepsilon_{\mathrm u}\right),\\
    &
    x(0)\in\mathcal X_{\mathrm i}(r_{\mathrm i}),\qquad
    x(t_{\mathrm f})\in\mathcal X_{\mathrm f}(r_{\mathrm f}),\\
    & u(0)=u(t_{\mathrm f})=u_{\mathrm h}, \qquad  \eta_{\mathrm p}(0)=0, \\
    & y(0)=0, \qquad z(0)=1
\end{aligned}
\end{empheq}

\vspace{-0.2cm}
Figures~\ref{fig:di_until}--\ref{fig:until_station_margin} show the optimized trajectory, speed history, and charging-station margin for the \until{} example. The vehicle first diverts toward the charging station and only afterwards proceeds to the terminal state. The speed trace remains below the prescribed threshold $v_{\mathrm safe}$ before the visiting event, while also remaining below the global speed bound for the entire maneuver. The signed charging-station margin is negative before arrival, crosses zero at the station boundary, and becomes positive inside the charging region. The first zero crossing therefore identifies the witness time at which the charging event occurs, exactly as required by the \until{} semantics.

\begin{figure}[t]
\centerline{\includegraphics[scale=0.45]{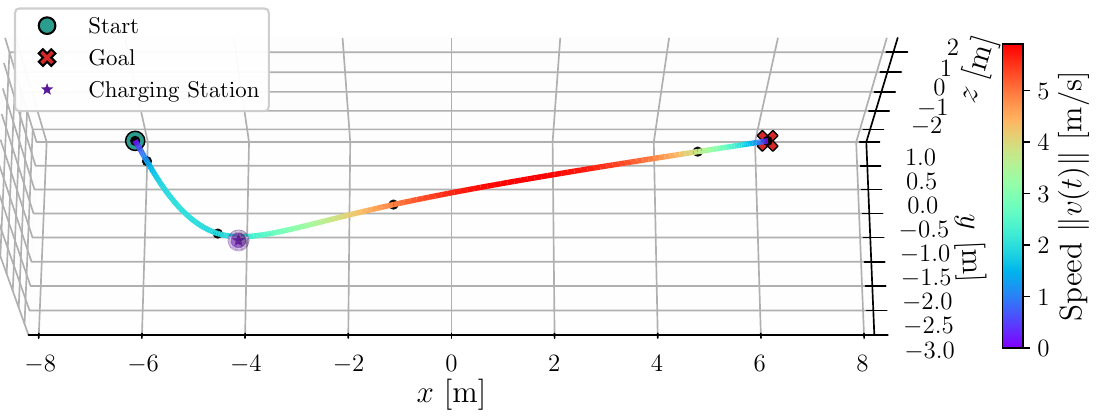}}
\caption{Optimized trajectory for the charging-station \until{} example. The vehicle first reaches the charging-station region and then proceeds to the terminal state, consistent with the continuous-time \until{} specification. The trajectory is color-coded by speed.}
\label{fig:di_until}
\end{figure}

\begin{figure}[t]
\centerline{\includegraphics[scale=0.47]{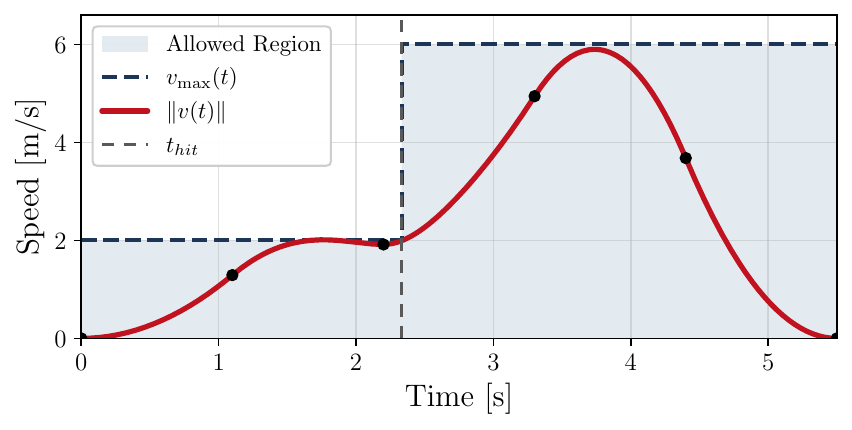}}
\caption{Speed profile for the charging-station \until{} example. The speed remains below the pre-charging threshold $v_{\mathrm safe}$ until the charging event occurs, while also remaining below the global speed bound throughout the maneuver.}
\label{fig:until_speed_profile}
\end{figure}

\begin{figure}[t]
\centerline{\includegraphics[scale=0.47]{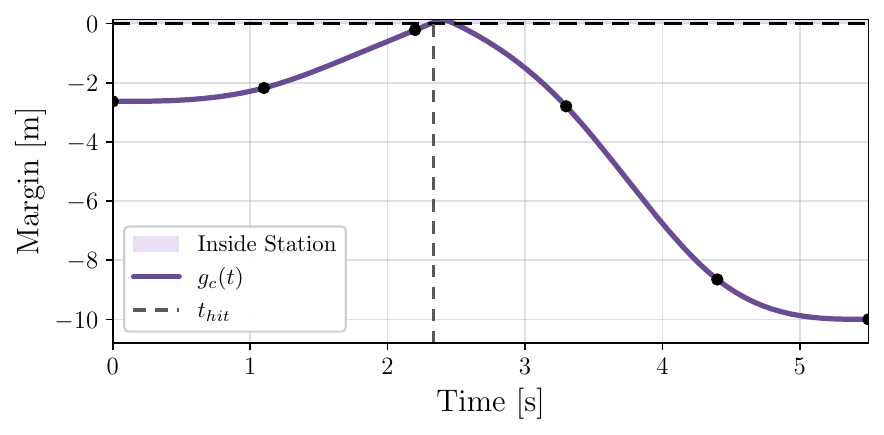}}
\caption{Signed charging-station margin for the \until{} example. The margin is negative outside the charging region, zero on its boundary, and positive inside. The first zero crossing marks the witness time for the charging event.}
\label{fig:until_station_margin}
\end{figure}

\vspace{-0.2cm}
\subsection{Free-final-time $6$-DoF quadrotor example}\label{sec:num_quad}
\vspace{-0.2cm}
For the main example, the physical quadrotor state and control are
$
    x_q := (r,v,\phi,\theta,\psi,p_b,q_b,r_b)\in\mathbb R^{12},
    \,
    u_q := (\tau_{\phi},\tau_{\theta},\tau_{\psi},T)\in\mathbb R^4,
$
where $r=(r_x,r_y,r_z)$ and $v=(v_x,v_y,v_z)$ denote position and velocity, $\Omega := (\phi,\theta,\psi)$ are the Euler angles, $\omega_b := (p_b,q_b,r_b)$ are the body angular rates, $T$ is the thrust magnitude, and $(\tau_\phi,\tau_\theta,\tau_\psi)$ are the body torques. The dynamics are modeled by the standard Newton--Euler equations for a quadrotor, as in~\cite{luukkonen2011modelling}, and are written compactly as
\vspace{-0.2cm}
\begin{align*}
    \dot x_q(t)=f_q(x_q(t),u_q(t)).
\end{align*}
The quadrotor is also subject to the continuous-time bounds
$0.06\,g_0 \le T(t) \le 1.75\,g_0$, $|\tau_\phi(t)|$, $|\tau_\theta(t)|$, $|\tau_\psi(t)| \le 0.25$, together with
$\|v(t)\|_2 \le 20$, $-5 \le r_z(t) \le 10$, $|\phi(t)|,\;|\theta(t)| \le 45^\circ$, $|p_b(t)|$, $|q_b(t)|$, $|r_b(t)| \le 100^\circ/\mathrm{s}$.
In the implementation, these continuous-time state and control bounds are handled through an additional accumulated path-penalty state $\eta_{\mathrm q}$, analogous to the path-constraint treatment used in the double-integrator examples.

\vspace{-0.2cm}
The problem uses $K=32$ shooting nodes and an initial final-time guess of $10~\mathrm{s}$. The final time is free and is optimized through the time-dilation variables in the multiple-shooting transcription. In the reported solution, the flight time converges to
$
    t_{\mathrm f}^{\star}=7.325~\mathrm{s}.
$

\vspace{-0.2cm}
Define the compact boundary-condition sets
\vspace{-0.2cm}
\begin{align*}
    \mathcal X_{q,\mathrm i}
    &:=
    \left\{
        x_q :
        r=(0,-10,1),\;
        v=0,\;
        \Omega=0,\;
        \omega_b=0
    \right\},\\
    \mathcal X_{q,\mathrm f}
    &:=
    \left\{
        x_q :
        8 \le r_x \le 10,\;
        v=0,\;
        \phi=\theta=0,\;
        \omega_b=0
    \right\}.
\end{align*}
To represent the moving predicates autonomously, the implementation augments the state with an internal clock state $\vartheta$, a path-penalty state $\eta_{\mathrm q}$, six obstacle-violation states, and two eventually-type waypoint pairs:
\vspace{-0.2cm}
\begin{align*}
    x^{\varphi_{\mathrm{q}}}
    :=
    \big(
        x_q,\vartheta,\eta_{\mathrm q},
        \xi_1^{\ell},\xi_1^{u},\xi_2^{\ell},\xi_2^{u},\xi_3^{\ell},\xi_3^{u},
        y_1,z_1,y_2,z_2
    \big).
\end{align*}
The internal clock is governed by
\vspace{-0.3cm}
\begin{align*}
    \dot\vartheta(t)=1,
    \qquad
    \vartheta(0)=0,
\end{align*}
and is used to parameterize the moving waypoints and moving obstacle gaps.

\vspace{-0.2cm}
The two moving waypoint centers are
\vspace{-0.3cm}
\begin{align*}
    c_1(\vartheta)
    &:=
    \left( \! 
        -5,\;
        -16 \!+\! 11.5 \left(\sin\left( \frac{\pi \vartheta}{5.655}\!-\!\frac{\pi}{2}\right)\!+\!1\right),\;
        1
    \! \right),\\
    c_2(\vartheta)
    &:=
    \left( \! 
        -10,\;
        7 \!-\! 11.5 \left(\sin\left(\frac{\pi \vartheta}{4.38}\!-\!\frac{\pi}{2}\right)\!+\!1\right),\;
        1
    \! \right),
\end{align*}
both with radius $r_w=0.5$. For each waypoint $i=1,2$, define
\vspace{-0.2cm}
\begin{align*}
    \rho_i(r,\vartheta):=r_w^2-\|r-c_i(\vartheta)\|_2^2,
\end{align*}
and the predicates
\vspace{-0.3cm}
\begin{align*}
    \mu_{w,i}(r,\vartheta):=(\rho_i(r,\vartheta)\ge 0).
\end{align*}
The moving-waypoint task is therefore
\vspace{-0.2cm}
\begin{align*}
    \varphi_{\mathrm{wp}}
    :=
    \bm F_{[0,t_{\mathrm f}]}\mu_{w,1}
    \wedge
    \bm F_{[0,t_{\mathrm f}]}\mu_{w,2}.
\end{align*}

\vspace{-0.4cm}
As in the double-integrator \eventually{} example, each waypoint is encoded by a geometric-integral state $y_i$ and an additive state $z_i$:
\vspace{-0.2cm}
\begin{align*}
    \dot y_i(t)
    &=
    \frac{1}{t_{\mathrm f}}
    y_i(t)\log\!\Big(
        \varepsilon_{\mathrm e}
        +
        [\rho_i(r(t),\vartheta(t))]_-^2
    \Big),\\
    \dot z_i(t)
    &=
    \frac{1}{t_{\mathrm f}}
    [\rho_i(r(t),\vartheta(t))]_+^2,
\end{align*}
with
$
    y_i(0)=1,
    \, 
    z_i(0)=0,
    \, 
    i=1,2.
$
The associated waypoint surrogates are
\vspace{-0.2cm}
\begin{align*}
    \widehat{\Gamma}_{w,i}^{q}
    :=
    \sqrt{c_q^2+z_i(t_{\mathrm f})}
    -
    \sqrt{c_q^2+y_i(t_{\mathrm f})},
    \qquad i=1,2.
\end{align*}

\vspace{-0.4cm}
The moving obstacle field is represented by three $x$-bands,
$
    [a_1,b_1]=[0.5,1.5],
    \,
    [a_2,b_2]=[3.5,4.5],
    \,
    [a_3,b_3]=[6.5,7.5],
$
with time-varying gap boundaries
\vspace{-0.2cm}
\begin{align*}
    \Delta(\vartheta)&:=3(\sin(\vartheta-\pi/2)+1),\\
    \ell_1(\vartheta)&=-1-\Delta(\vartheta),\qquad u_1(\vartheta)=1-\Delta(\vartheta),\\
    \ell_2(\vartheta)&=-10-\Delta(\vartheta),\qquad u_2(\vartheta)=-8-\Delta(\vartheta),\\
    \ell_3(\vartheta)&=-1+\Delta(\vartheta),\qquad u_3(\vartheta)=1+\Delta(\vartheta).
\end{align*}
The corresponding implication-type obstacle specifications are
\vspace{-0.2cm}
\begin{align*}
    \varphi_{\mathrm{obs},j}^{\ell}(r,\vartheta)
    &:=
    \big(a_j\le r_x\le b_j \Rightarrow r_y\ge \ell_j(\vartheta)\big),\\
    \varphi_{\mathrm{obs},j}^{u}(r,\vartheta)
    &:=
    \big(a_j\le r_x\le b_j \Rightarrow r_y\le u_j(\vartheta)\big),
    \qquad 
\end{align*}
for $j=1,2,3$. Hence,
\vspace{-0.2cm}
\begin{align*}
    \varphi_{\mathrm{obs}}
    :=
    \bm G_{[0,t_{\mathrm f}]}
    \Bigg(
        \bigwedge_{j=1}^{3}
        \varphi_{\mathrm{obs},j}^{\ell}
        \wedge
        \varphi_{\mathrm{obs},j}^{u}
    \Bigg).
\end{align*}
In the implementation, these six moving implication constraints are encoded by the auxiliary states
\vspace{-0.3cm}
\begin{align*}
    \dot \xi_j^{\ell}(t)
    &=
    \chi_j^{\ell}(r(t),\vartheta(t)),
    \qquad
    \xi_j^{\ell}(0)=0, \\
    \dot \xi_j^{u}(t)
    &=
    \chi_j^{u}(r(t),\vartheta(t)),
    \qquad
    \xi_j^{u}(0)=0,
\end{align*}
for $j=1,2,3$, where
\vspace{-0.3cm}
\begin{align*}
    \chi_j^{\ell}(r,\vartheta)
    &:=
    [a_j-r_x]_-^2\,[r_x-b_j]_-^2\,[r_y-\ell_j(\vartheta)]_-^2,\\
    \chi_j^{u}(r,\vartheta)
    &:=
    [a_j-r_x]_-^2\,[r_x-b_j]_-^2\,[u_j(\vartheta)-r_y]_-^2.
\end{align*}
Each function $\chi_j^{\ell}$ or $\chi_j^{u}$ vanishes when the corresponding moving implication constraint is satisfied, and becomes positive only under violation. For compactness, define the total accumulated moving-obstacle violation as
\vspace{-0.2cm}
\begin{align*}
    \Xi_{\mathrm{obs}}(t)
    :=
    \sum_{j=1}^{3}\big(\xi_j^{\ell}(t)+\xi_j^{u}(t)\big).
\end{align*}
We then define the corresponding obstacle-avoidance surrogate
\vspace{-0.2cm}
\begin{align*}
    \widehat{\Gamma}_{\mathrm{obs}}:=-\Xi_{\mathrm{obs}}(t_{\mathrm f}).
\end{align*}
Combining the temporal requirements, the main logical task for the quadrotor is
\vspace{-0.2cm}
\begin{align*}
    \varphi_{\mathrm{q}}
    :=
    \varphi_{\mathrm{obs}} \wedge \varphi_{\mathrm{wp}},
\end{align*}
together with the terminal end-zone constraint encoded in $\mathcal X_{q,\mathrm f}$.

\vspace{-0.2cm}
Finally, define the overall terminal surrogate
\vspace{-0.2cm}
\begin{align*}
    \widehat{\Gamma}_{\mathrm{q}}
    :=
    \widehat{\Gamma}_{w,1}^{q}
    +
    \widehat{\Gamma}_{w,2}^{q}
    +
    w_{\Xi}\widehat{\Gamma}_{\mathrm{obs}}
    -
    w_{\eta}\eta_{\mathrm q}(t_{\mathrm f})
    -
    w_t t_{\mathrm f}.
\end{align*}

\vspace{-0.3cm}
The resulting free-final-time problem is written as
\vspace{-0.2cm}
\begin{empheq}[box=\fbox]{equation*}
\begin{aligned}
    \underset{x^{\varphi_{\mathrm{q}}},u_q,t_{\mathrm f}}{\mathrm{maximize}}
    \;\; &
    \widehat{\Gamma}_{\mathrm{q}} \\
    \mathrm{subject\ to}\;\;
    &
    \dot x_q(t)=f_q(x_q(t),u_q(t)),\\
    &
    \dot\vartheta(t)=1,\\
    &
    \dot \eta_{\mathrm q}(t)=\chi_{\mathrm{path},q}(x_q(t),u_q(t)),\\
    &
    \dot y_i(t) \!=\! \frac{y_i(t)}{t_{\mathrm f}}\log\!\left(\!\varepsilon_{\mathrm e}\!+\![\rho_i(r(t),\vartheta(t))]_-^2\!\right),
    \forall i,
    \\
    &
    \dot z_i(t)=\frac{1}{t_{\mathrm f}}[\rho_i(r(t),\vartheta(t))]_+^2,\quad \forall i,\\
    &
    \dot \xi_j^{\ell}(t)=\chi_j^{\ell}(r(t),\vartheta(t)),\quad \forall j,\\
    &
    \dot \xi_j^{u}(t)=\chi_j^{u}(r(t),\vartheta(t)),\quad \forall j,\\
    &
    x_q(0)\in\mathcal X_{q,\mathrm i},\quad
    x_q(t_{\mathrm f})\in\mathcal X_{q,\mathrm f},\\
    &
    \vartheta(0)=0, \quad \eta_{\mathrm q}(0)=0, \\
    &
    y_i(0)=1,\quad z_i(0)=0, \quad
    \forall i, \\
    &
    \xi_j^{\ell}(0)=0,\quad \xi_j^{u}(0)=0,\quad \forall j.
\end{aligned}
\end{empheq}

\vspace{-0.2cm}
The quadrotor example is the main numerical demonstration of the paper. Here, the vehicle must coordinate its motion with both the moving waypoint regions and the moving obstacle gaps while simultaneously satisfying nonlinear rigid-body dynamics, terminal-set constraints, and state/control bounds. Since both the waypoint centers and the obstacle boundaries vary continuously in time, satisfaction only at the shooting nodes would be insufficient, and continuous-time evaluation is therefore essential in this example.

\vspace{-0.2cm}
Figure~\ref{fig:qf_traj} shows that the optimized $3$D trajectory first intercepts the second moving waypoint region at approximately $t=2.26~\mathrm s$, then the first moving waypoint region at approximately $t=3.11~\mathrm s$, and subsequently passes through the three moving doorway gaps at approximately $t=3.79~\mathrm s$, $t=4.98~\mathrm s$, and $t=6.38~\mathrm s$ before reaching the terminal end zone. This ordering is not imposed a priori; rather, it emerges from the optimizer’s need to synchronize the trajectory with the moving waypoint and obstacle geometry. In the reported solution, the flight-time variable is initialized at $10~\mathrm s$ and converges to $7.325~\mathrm s$, showing that the optimizer adapts the overall traversal time in order to coordinate the maneuver with the oscillatory target and obstacle dynamics. The final continuous-time trajectory satisfies both moving-waypoint visitation requirements and avoids all six moving obstacle constraints throughout the horizon.

\vspace{-0.2cm}
Figure~\ref{fig:qf_states} confirms that the optimized trajectory remains within the prescribed thrust, torque, speed, altitude, attitude, and body-rate bounds over the continuous-time trajectory representation.

\begin{figure}[t]
\centerline{\includegraphics[scale=0.30]{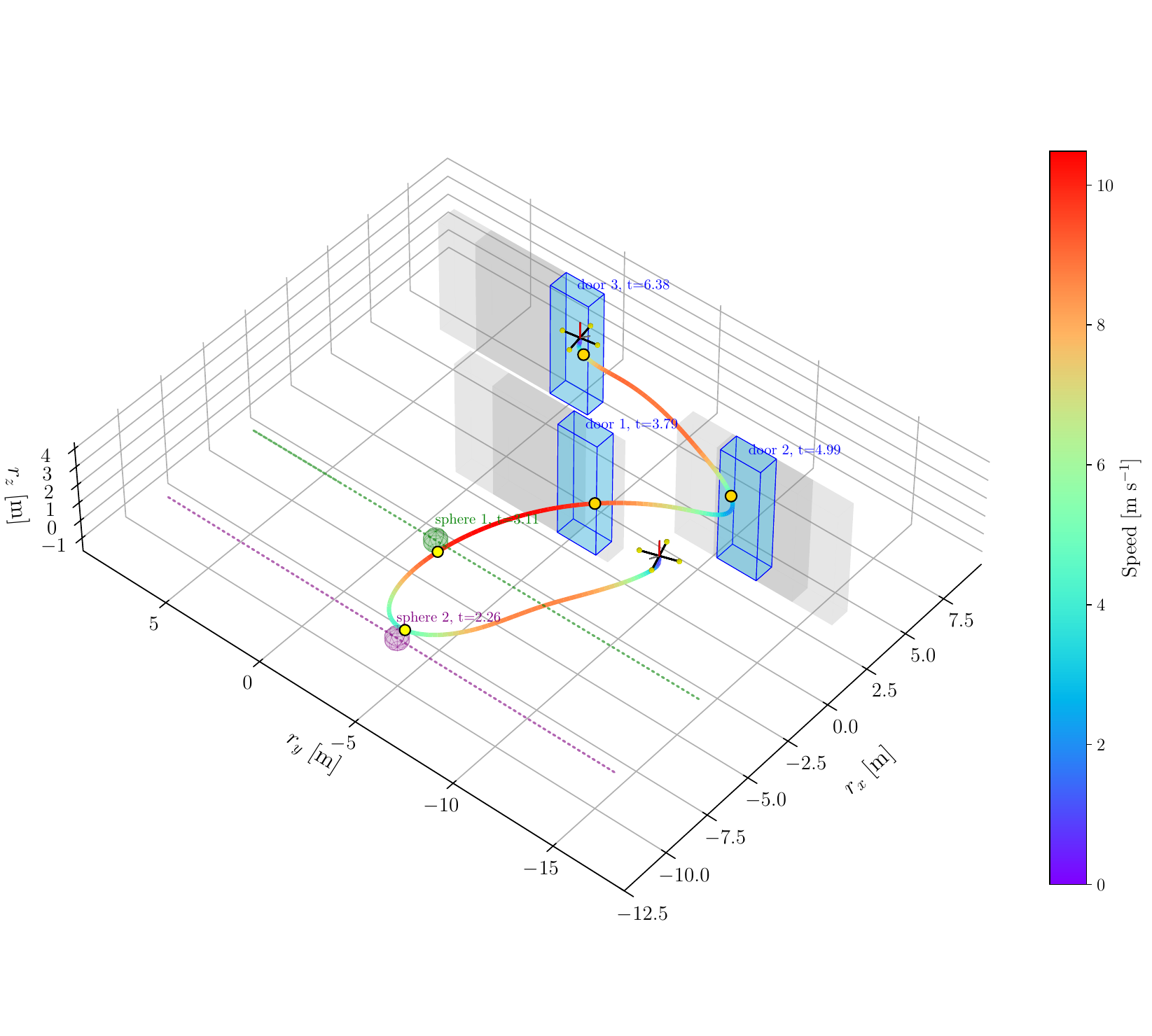}}
\caption{Optimized $3$D trajectory for the free-final-time $6$-DoF quadrotor example. The quadrotor first intercepts the two moving waypoint regions and then passes through the three moving doorway gaps while avoiding the shaded obstacle regions before reaching the terminal end zone. The annotated snapshots indicate the corresponding event times, and the trajectory is color-coded by speed.}
\label{fig:qf_traj}
\end{figure}

\begin{figure}[t]
\centerline{\includegraphics[scale=0.45]{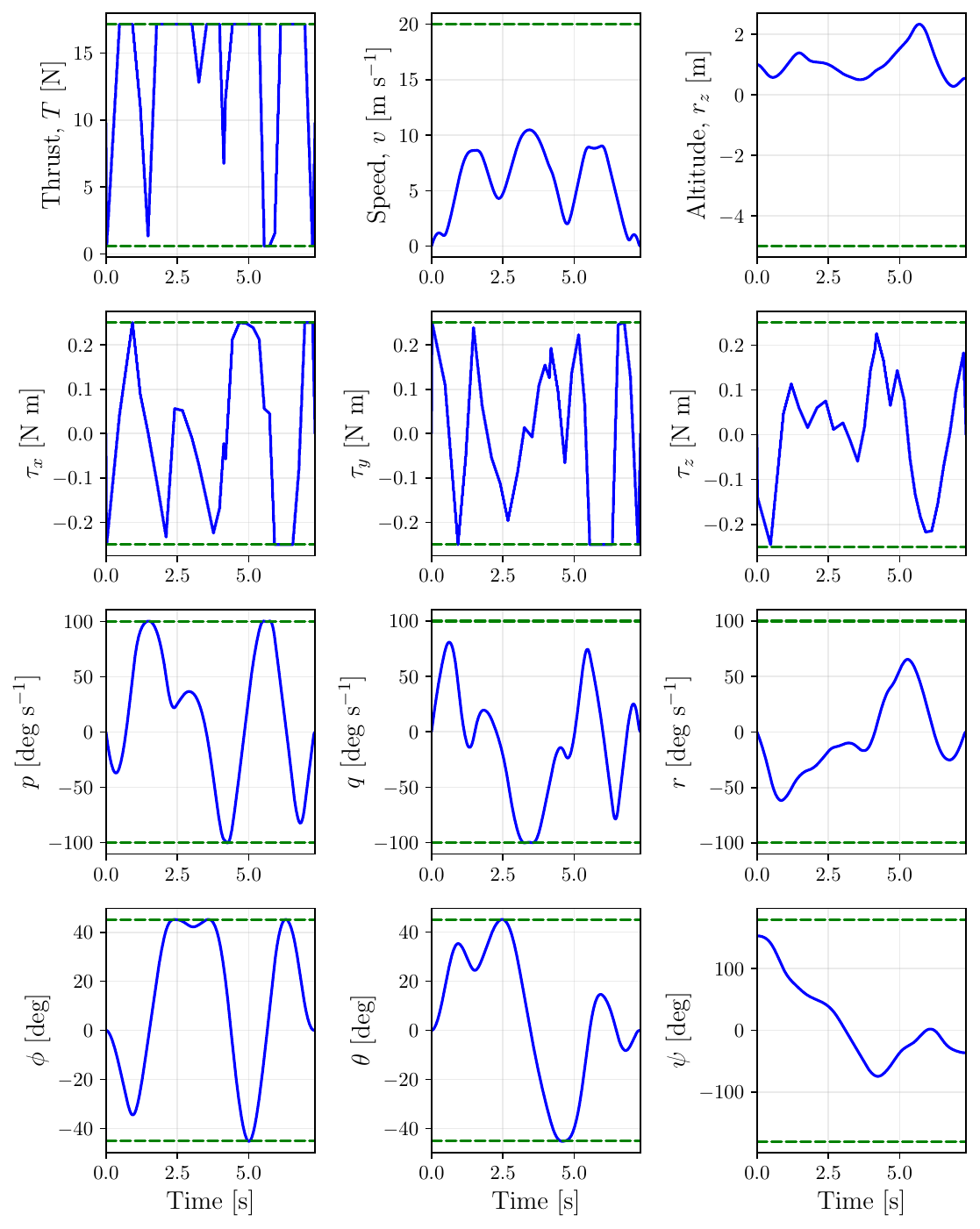}}
\caption{Representative state and input histories for the free-final-time $6$-DoF quadrotor example, showing that the optimized continuous-time trajectory remains within the prescribed thrust, torque, speed, altitude, attitude, and body-rate bounds.}
\label{fig:qf_states}
\end{figure}

\vspace{-0.2cm}
\section{Conclusion}
\vspace{-0.2cm}

This paper presented a successive convexification framework for trajectory optimization with continuous-time satisfaction of Signal Temporal Logic specifications. Built on the smooth and exact GMSR parameterization, the framework combines augmentation-based CT-STL modeling with time-dilation, finite-dimensional control parameterization, multiple-shooting discretization, and a convergence-guaranteed prox-convex SCP scheme. The proposed realization embeds temporal aggregation into auxiliary continuous-time dynamics, making the temporal-logic sensitivities compatible with the same multiple-shooting and automatic-differentiation machinery used for the original dynamics. An exact dense-time realization was also summarized as a complementary construction and related to the prior dense-time formulation in~\cite{uzun2026smooth}.

\vspace{-0.2cm}
The numerical examples showed that the proposed approach can handle continuous-time \always{}, \eventually{}, and \until{} specifications, as well as a more challenging free-final-time $6$-DoF quadrotor problem with moving waypoints and moving obstacle constraints. Overall, the framework provides a smooth and optimization-friendly way to incorporate rich continuous-time temporal logic requirements into nonlinear trajectory generation.
{\RaggedRight%
\bibliographystyle{unsrturl}      
\bibliography{ref_dis}%
}
\section{Appendix}

\vspace{-0.2cm}
\subsection{Feasibility of an FOH parameterization of the dilation factor}
\label{app:foh_dilation}
\vspace{-0.2cm}
\begin{prop}[Why ZOH is used for $s(\tau)$]
\label{prop:foh_dilation}
If the dilation factor were parameterized by FOH, then the physical interval lengths would satisfy
\vspace{-0.2cm}
\begin{align*}
    \Delta t_k=\frac12(s_k+s_{k+1})\Delta\tau.
\end{align*}
As a result, the normalized interval-length sequence $\Delta_k=\Delta t_k/\Delta\tau$ would have to satisfy a coupled linear system $As=\Delta$ with $s\ge 0$. Such a system is feasible if and only if every alternating sum over every contiguous odd-length block of $\Delta$ is nonnegative, namely
\vspace{-0.2cm}
\begin{align*}
    \sum_{r=i}^{j}(-1)^{r-i}\Delta_r \ge 0,
    \quad
    \forall\, 1\le i\le j\le K-1,\;\; j-i\ \text{even}.
\end{align*}
Hence, FOH imposes a structural restriction on the realizable physical interval lengths, so not every positive interval-length sequence can be represented by nonnegative nodal dilation values.
\end{prop}
\begin{proof}
Suppose the dilation factor is parameterized by FOH. Then
\vspace{-0.2cm}
\begin{align*}
    \Delta t_k = \frac12(s_k+s_{k+1})\Delta\tau,
    \qquad k=1,\dots,K-1.
\end{align*}
With $\Delta_k:=\Delta t_k/\Delta\tau$, realizability of a prescribed interval-length sequence is equivalent to
\vspace{-0.2cm}
\begin{align}
    As=\Delta,\qquad s\ge 0,
    \label{eq:foh_primal}
\end{align}
where
\vspace{-0.2cm}
\[
A=
\begin{bmatrix}
\frac12&\frac12&0&\cdots&0\\
0&\frac12&\frac12&\ddots&\vdots\\
\vdots&&\ddots&\ddots&0\\
0&\cdots&0&\frac12&\frac12
\end{bmatrix}
\in\mathbb R^{(K-1)\times K}.
\]
By Farkas' lemma, \eqref{eq:foh_primal} is feasible if and only if
\vspace{-0.2cm}
\begin{align}
    \Delta^\top y \le 0
    \qquad
    \forall\, y \in \mathbb R^{K-1}\ \text{s.t.}\ A^\top y \le 0.
    \label{eq:farkas_foh}
\end{align}
The inequalities $A^\top y\le 0$ are
$
    y_1 \le 0,
    \,
    y_{i-1}+y_i \le 0 \ \ (2\le i\le K-1),
    \,
    y_{K-1}\le 0.
$
The cone defined by these inequalities is polyhedral, and its extreme rays are supported on contiguous odd-length blocks with alternating signs:
\vspace{-0.2cm}
\begin{align*}
    y_k^{(i,j)}=
    \begin{cases}
        (-1)^{k-i+1}, & i\le k\le j,\\
        0, & \text{otherwise},
    \end{cases}
\end{align*}
$1\le i\le j\le K-1,\ \ j-i\ \text{even}.$ Therefore, it suffices to check \eqref{eq:farkas_foh} on these extreme rays, which yields
\vspace{-0.2cm}
\begin{align*}
    \Delta^\top y^{(i,j)} \le 0
    \quad\Longleftrightarrow\quad
    \sum_{r=i}^{j}(-1)^{r-i}\Delta_r \ge 0,
\end{align*}
$\forall\, 1\le i\le j\le K-1,\ \ j-i\ \text{even}$. This proves Proposition~\ref{prop:foh_dilation}.
\end{proof}


For example, the positive interval-length sequence $\Delta=[1,\,3,\,1]$ is infeasible under FOH since $1-3+1=-1<0$. Under ZOH, by contrast, the interval lengths are fully decoupled through $\Delta t_k=s_k\Delta\tau$, so any positive interval-length sequence can be represented directly. This distinction is also important numerically for optimization methods. Under FOH, adjacent interval lengths share nodal variables, so changing a single local time step generally requires coordinated updates across a larger portion of the sequence $\{s_k\}_{k=1}^{K-1}$. This reduces the optimizer's ability to adjust the temporal mesh locally and can slow convergence in practice. Under ZOH, each interval duration is decoupled from its neighbors, allowing local time allocations to be adjusted independently.

\vspace{-0.2cm}
\subsection{Linearization of the multiple-shooting defects}
\label{app:flow_map_linearization}
\vspace{-0.2cm}
For gradient-based optimization, one also needs the Jacobians of the flow map and of the multiple-shooting defect constraints. These Jacobians can be computed by integrating the associated variational equations alongside the state trajectory on each shooting interval.

\vspace{-0.2cm}
Recall that the physical control is parameterized by FOH, whereas the time-dilation factor is parameterized by ZOH. On interval $[\tau_k,\tau_{k+1}]$, define
\vspace{-0.2cm}
\begin{align*}
    \lambda_k^{-}(\tau)
    &:=
    \frac{\tau_{k+1}-\tau}{\Delta\tau},
    &
    \lambda_k^{+}(\tau)
    &:=
    \frac{\tau-\tau_k}{\Delta\tau},
    \;\;
    \tau\in[\tau_k,\tau_{k+1}],
\end{align*}
so that
\vspace{-0.2cm}
\begin{align*}
    u(\tau)
    =
    \lambda_k^{-}(\tau)u_k
    +
    \lambda_k^{+}(\tau)u_{k+1},
    \qquad
    s(\tau)=s_k.
\end{align*}
The corresponding hybrid input is
\vspace{-0.2cm}
\begin{align*}
    \tilde u(\tau)
    =
    \begin{bmatrix}
        u(\tau)\\
        s_k
    \end{bmatrix}.
\end{align*}
Along the interval trajectory, define
\vspace{-0.2cm}
\begin{align*}
    A_k(\tau)
    &:=
    \nabla_{\tilde x} f\big(\tilde x(\tau),\tilde u(\tau)\big),
    \\
    B_k(\tau)
    &:=
    \nabla_{u} f\big(\tilde x(\tau),\tilde u(\tau)\big),
    \\
    S_k(\tau)
    &:=
    \partial_{s} f\big(\tilde x(\tau),\tilde u(\tau)\big).
\end{align*}
The sensitivities of the propagated state with respect to the active local variables on interval $k$ are obtained from the linear time-varying initial-value problems
\vspace{-0.2cm}
\begin{align*}
    \frac{d}{d\tau}\Phi_k^{x}(\tau)
    &\!=\!
    A_k(\tau)\Phi_k^{x}(\tau),
    &
    \Phi_k^{x}(\tau_k)
    &\!=\!
    I,
    \nonumber\\
    \frac{d}{d\tau}\Phi_k^{-}(\tau)
    &\!=\!
    A_k(\tau)\Phi_k^{-}(\tau)
    \!+\!
    B_k(\tau)\lambda_k^{-}(\tau),
    &
    \Phi_k^{-}(\tau_k)
    &\!=\!
    0,
    \nonumber\\
    \frac{d}{d\tau}\Phi_k^{+}(\tau)
    &\!=\!
    A_k(\tau)\Phi_k^{+}(\tau)
    \!+\!
    B_k(\tau)\lambda_k^{+}(\tau),
    &
    \Phi_k^{+}(\tau_k)
    &\!=\!
    0,
    \nonumber\\
    \frac{d}{d\tau}\Phi_k^{s}(\tau)
    &\!=\!
    A_k(\tau)\Phi_k^{s}(\tau)
    \!+\!
    S_k(\tau),
    &
    \Phi_k^{s}(\tau_k)
    &\!=\!
    0,
\end{align*}
for $\tau\in[\tau_k,\tau_{k+1}]$. Their terminal values yield the Jacobians of the flow map:
\vspace{-0.2cm}
\begin{align*}
    \bar A_k
    &:=
    \Phi_k^{x}(\tau_{k+1})
    =
    \partial_{\tilde x_k}
    f_{\tau_k}^{\tau_{k+1}}(\tilde x_k,u_k,u_{k+1},s_k),
    \nonumber\\
    \bar B_k^{-}
    &:=
    \Phi_k^{-}(\tau_{k+1})
    =
    \partial_{u_k}
    f_{\tau_k}^{\tau_{k+1}}(\tilde x_k,u_k,u_{k+1},s_k),
    \nonumber\\
    \bar B_k^{+}
    &:=
    \Phi_k^{+}(\tau_{k+1})
    =
    \partial_{u_{k+1}}
    f_{\tau_k}^{\tau_{k+1}}(\tilde x_k,u_k,u_{k+1},s_k),
    \nonumber\\
    \bar S_k
    &:=
    \Phi_k^{s}(\tau_{k+1})
    =
    \partial_{s_k}
    f_{\tau_k}^{\tau_{k+1}}(\tilde x_k,u_k,u_{k+1},s_k).
\end{align*}
Consequently, the Jacobians of the defect constraint
$
D_k(Z)=\tilde x_{k+1}
-
f_{\tau_k}^{\tau_{k+1}}(\tilde x_k,u_k,u_{k+1},s_k)
$
are
\vspace{-0.2cm}
\begin{align*}
    \partial_{\tilde x_{k+1}} D_k(Z)
    &=
    I,
    &
    \partial_{\tilde x_k} D_k(Z)
    &=
    -\bar A_k,
    \\
    \partial_{u_k} D_k(Z)
    &=
    -\bar B_k^{-},
    &
    \partial_{u_{k+1}} D_k(Z)
    &=
    -\bar B_k^{+},
    \\
    \partial_{s_k} D_k(Z)
    &=
    -\bar S_k.
\end{align*}
If one works directly with the hybrid nodal inputs $\tilde u_k$ and $\tilde u_{k+1}$, the corresponding Jacobians are obtained by assembling these blocks with zero blocks for inactive components, since $s_{k+1}$ does not affect the propagation on interval $k$ under the ZOH parameterization.

The same variational equations may also be integrated only up to an intermediate integration subnode $\tau\in[\tau_k,\tau_{k+1}]$. This yields the Jacobians of the intermediate subnode states with respect to the local variables $\tilde x_k$, $u_k$, $u_{k+1}$, and $s_k$. This observation is useful for the dense-time realization summarized in Appendix~\ref{ssec:exact_dense_ctstl}, where STL robustness values are evaluated on integration subnodes and differentiated with respect to the original nodal decision variables.

\vspace{-0.2cm}
\subsection{Exact dense-time realization on the integration grid}\label{ssec:exact_dense_ctstl}
\vspace{-0.2cm}
For completeness, and following the dense-time CT-STL construction of~\cite{uzun2026smooth}, we summarize an exact realization of the CT-STL temporal operators on the dense integration grid induced by the dynamics discretization. In this realization, no operator-specific temporal auxiliary states are introduced. Instead, the CT-STL semantics are evaluated directly on the same dense numerical trajectory representation already used to integrate the dynamics. This yields an exact realization on the chosen dense-time trajectory representation, up to the numerical accuracy of the underlying integration scheme, without introducing any additional approximation parameter.

\vspace{-0.2cm}
Under time-dilation, both the predicates and the temporal intervals are mapped from the physical-time domain to the dilated-time domain. In this subsection, $\tilde x(\tau)$ denotes the time-dilated system trajectory on which the temporal operators are evaluated. For example, if
\vspace{-0.2cm}
\begin{align*}
    (x,t)\models \varphi_1 \bm U_{[a,b]} \varphi_2,
\end{align*}
then, at the corresponding dilated time $\tau$ satisfying $\tilde t(\tau)=t$,
\vspace{-0.2cm}
\begin{align*}
    (\tilde x,\tau)\models \tilde\varphi_1 \bm U_{[\tau_a,\tau_b]} \tilde\varphi_2,
\end{align*}
where $\tilde\varphi_i := (\tilde g_i(\tilde x(\tau))\ge 0)$, $\tilde g_i(\tilde x(\tau)) := g_i(x(t))$, and the dilated relative bounds satisfy
\vspace{-0.2cm}
\begin{align*}
    \tilde t(\tau+\tau_a)=t+a,
    \qquad
    \tilde t(\tau+\tau_b)=t+b.
\end{align*}
The dilated bounds generally depend on the evaluation time and on the time-dilation trajectory, but for any fixed strictly increasing time reparameterization they are uniquely defined by the equations above. Thus, the CT-STL semantics are preserved exactly under the time reparameterization.

\vspace{-0.2cm}
To realize the temporal operators exactly on the numerical trajectory representation, we reuse the intermediate integration subnodes generated during the numerical evaluation of the flow map on each shooting interval $[\tau_k,\tau_{k+1}]$, $k\in\{1,\dots,K-1\}$. Let
\vspace{-0.2cm}
\[
\tau_k=\tau_k^0<\tau_k^1<\cdots<\tau_k^N=\tau_{k+1}
\]
denote the integration subnodes associated with $[\tau_k,\tau_{k+1}]$. The corresponding dense-time states are obtained by applying a chosen numerical integration scheme over these subintervals, for example, a Runge--Kutta method. Consequently, each state $\tilde{x}(\tau_k^i)$ is an explicitly generated point on the numerical trajectory induced by the nodal variables and the chosen input parameterization. Flattening all such subnodes over the horizon yields the dense-time grid
\vspace{-0.2cm}
\[
\{\bar\tau_m\}_{m=1}^{M},
\qquad
M=(K-1)N+1,
\]
and we denote the associated dense-time states by
\vspace{-0.2cm}
\[
\bar x_m:=\tilde x(\bar\tau_m).
\]
Once these dense-time states are available from the dynamics discretization, CT-STL evaluation requires only predicate evaluations and GMSR compositions, with no additional state propagation.

\vspace{-0.2cm}
The following schematic illustrates how the shooting-node states $\tilde x_k$, the intermediate subnodes $\tau_k^i$, and the flattened dense-time sequence $\bar x_m=\tilde x(\bar\tau_m)$ are related.
\vspace{-0.4cm}
\begin{center}
\begin{tikzpicture}[x=1.0cm,y=0.7cm,>=latex]

\definecolor{mainblue}{RGB}{40,90,160}
\definecolor{suborange}{RGB}{220,130,40}
\definecolor{densegreen}{RGB}{20,120,80}

\draw[thick] (0,0) -- (8,0);

\foreach \x in {0,1.0,2.0,2.7,3.8,4.8,5.8,6.5,8.0}
    \fill (\x,0) circle (1.2pt);

\node[above=3pt, text=mainblue] at (0,0) {$\tilde{x}_1$};
\node[above=3pt, text=mainblue] at (3.8,0) {$\tilde{x}_2$};
\node[above=3pt, text=mainblue] at (8.0,0) {$\tilde{x}_K$};

\node[above=3pt, text=suborange] at (1.0,0) {$\tau_1^1$};
\node[above=3pt, text=suborange] at (2.0,0) {$\tau_1^2$};
\node[above=3pt] at (2.7,0) {$\cdots$};
\node[above=3pt, text=suborange] at (4.8,0) {$\tau_2^1$};
\node[above=3pt, text=suborange] at (5.8,0) {$\tau_2^2$};
\node[above=3pt] at (6.5,0) {$\cdots$};

\node[below=3pt, text=densegreen] at (0,0) {$\bar{x}_1$};
\node[below=3pt, text=densegreen] at (1.0,0) {$\bar{x}_2$};
\node[below=3pt, text=densegreen] at (2.0,0) {$\bar{x}_3$};
\node[below=3pt] at (2.7,0) {$\cdots$};
\node[below=3pt, text=densegreen] at (3.8,0) {$\bar{x}_{N+1}$};
\node[below=3pt, text=densegreen] at (4.8,0) {$\bar{x}_{N+2}$};
\node[below=3pt, text=densegreen] at (5.8,0) {$\bar{x}_{N+3}$};
\node[below=3pt] at (6.5,0) {$\cdots$};
\node[below=3pt, text=densegreen] at (8.0,0) {$\bar{x}_M$};

\end{tikzpicture}
\end{center}
\vspace{-0.2cm}

Let
\[
\tilde X:=\{\tilde x_k\}_{k=1}^{K},
\qquad
\tilde U:=\{\tilde u_k\}_{k=1}^{K},
\]
and write
\[
\tilde{\Gamma}_{\bm c}^{\tilde\varphi}(\tilde X,\tilde U;\bar\tau_m)
\]
for the robustness value of a transformed subformula $\tilde\varphi$ at dense-time node $\bar\tau_m$, where the dependence on $(\tilde X,\tilde U)$ enters through the dense numerical trajectory. For an atomic predicate $\tilde\mu:=(\tilde g(\tilde x)\ge 0)$,
\vspace{-0.2cm}
\[
\tilde{\Gamma}_{\bm c}^{\tilde\mu}(\tilde X,\tilde U;\bar\tau_m):=\tilde g(\bar x_m).
\]
For composite formulas, the values $\tilde{\Gamma}_{\bm c}^{\tilde\varphi}(\tilde X,\tilde U;\bar\tau_m)$ are obtained recursively from the logical constructions of the previous subsection.

\vspace{-0.2cm}
Now consider a transformed STL formula evaluated at dense-time node $\bar\tau_\ell$. Suppose the original physical-time specification uses the interval $[a,b]$. Let $[\tau_a,\tau_b]$ denote the corresponding dilated relative interval at $\bar\tau_\ell$, as defined above. We then define
\vspace{-0.2cm}
\[
\mathcal I_\ell[a,b]
:=
\left\{
m\in\{1,\dots,M\}
:\;
\bar\tau_m\in[\bar\tau_\ell+\tau_a,\ \bar\tau_\ell+\tau_b]
\right\},
\]
assuming that this interval lies within the time horizon. If the boundary points are not already included in the dense-time grid, they can be inserted by refining the corresponding integration subintervals.

\vspace{-0.2cm}
Using this notation, the temporal robustness measures are defined by
\vspace{-0.2cm}
\begin{align*}
    \tilde{\Gamma}_{\bm c}^{\bm F_{[\tau_a,\tau_b]}\tilde\varphi}(\tilde X,\tilde U;\bar\tau_\ell)
    &:=
    {}^{\vee} h^{c_1}
    \Big(
        [\,\tilde{\Gamma}_{\bm c}^{\tilde\varphi}(\tilde X,\tilde U;\bar\tau_m)\,]_{m\in\mathcal I_\ell[a,b]}
    \Big), \\
    \tilde{\Gamma}_{\bm c}^{\bm G_{[\tau_a,\tau_b]}\tilde\varphi}(\tilde X,\tilde U;\bar\tau_\ell)
    &:=
    {}^{\wedge} h^{c_1}
    \Big(
        [\,\tilde{\Gamma}_{\bm c}^{\tilde\varphi}(\tilde X,\tilde U;\bar\tau_m)\,]_{m\in\mathcal I_\ell[a,b]}
    \Big).
\end{align*}
Accordingly,
\vspace{-0.2cm}
\begin{align*}
    (\tilde x,\bar\tau_\ell)\models \bm F_{[\tau_a,\tau_b]}\tilde\varphi
    &\iff
    \tilde{\Gamma}_{\bm c}^{\bm F_{[\tau_a,\tau_b]}\tilde\varphi}(\tilde X,\tilde U;\bar\tau_\ell)\ge 0, \\
    (\tilde x,\bar\tau_\ell)\models \bm G_{[\tau_a,\tau_b]}\tilde\varphi
    &\iff
    \tilde{\Gamma}_{\bm c}^{\bm G_{[\tau_a,\tau_b]}\tilde\varphi}(\tilde X,\tilde U;\bar\tau_\ell)\ge 0.
\end{align*}

Similarly, the \until{} operator is parameterized as
\vspace{-0.2cm}
\begin{align*}
    \tilde{\Gamma}_{\bm c}^{\tilde\varphi_1 \bm U_{[\tau_a,\tau_b]} \tilde\varphi_2}(\tilde X,\tilde U;\bar\tau_\ell)
    &:=
    {}^{\vee} h^{c_1}
    \Big(
        [\,z_m\,]_{m\in\mathcal I_\ell[a,b]}
    \Big),
\end{align*}
where
\vspace{-0.2cm}
\begin{align*}
    z_m
    &:=
    {}^{\wedge} h^{c_2}
    \Big(
        \tilde{\Gamma}_{\bm c}^{\tilde\varphi_2}(\tilde X,\tilde U;\bar\tau_m),\;
        {}^{\wedge} h^{c_3}
        \big(
            [\,\tilde{\Gamma}_{\bm c}^{\tilde\varphi_1}(\tilde X,\tilde U;\bar\tau_q)\,]_{q=\ell}^{m}
        \big)
    \Big).
\end{align*}
Hence,
\vspace{-0.2cm}
\begin{align*}
    (\tilde x,\bar\tau_\ell)\models \tilde\varphi_1 \bm U_{[\tau_a,\tau_b]} \tilde\varphi_2
    &\iff
    \tilde{\Gamma}_{\bm c}^{\tilde\varphi_1 \bm U_{[\tau_a,\tau_b]} \tilde\varphi_2}(\tilde X,\tilde U;\bar\tau_\ell)\ge 0.
\end{align*}
Thus, $\bm G$ requires satisfaction at all dense-time samples in the interval, $\bm F$ requires satisfaction at at least one such sample, and $\bm U$ requires the existence of a sample at which $\tilde\varphi_2$ holds while $\tilde\varphi_1$ holds at all preceding samples.

\vspace{-0.2cm}
For a fixed dense grid and fixed active index sets, the resulting parameterization is smooth with respect to the optimization variables, since the dependence on the decision variables enters through the integrated dense-time states and the smooth GMSR operators. Moreover, because the temporal operators are evaluated on the same dense-time grid used to discretize the dynamics, the CT-STL semantics are enforced exactly on the dense numerical trajectory representation, up to the numerical accuracy of the underlying integration scheme.

\end{document}